\documentclass[11pt]{article}
\usepackage{amsmath,amsthm}
\usepackage{geometry}                
\usepackage{graphicx}
\usepackage{amssymb}
\usepackage{epstopdf}
\usepackage{caption}
\newtheorem{theorem}{Theorem}
\newtheorem{corollary}[theorem]{Corollary}
\newtheorem{lemma}[theorem]{Lemma}
\newtheorem{proposition}[theorem]{Proposition}

\title{A Combinatorial Algorithm for Computing Higher Order Linking Numbers}
\author{Chun-Chung Hsieh,  Louis Kauffman, {\it and} Chichen M. Tsau\\{Institute of Mathematics Academia Sinica}, Department of Mathematics  University of Illinois \\at Chicago, {\it and} Department of Mathematics and Computer Science 
Saint Louis University\\{\it cchsieh@math.sinica.edu.tw}, {\it Kauffman@uic.edu}, and 
{\it tsaumc@slu.edu}}

\date{}                                           

\begin{document}
\maketitle

\begin{abstract} 
We develop the intersection theory at relative chain-cochain level, and apply it along with the use of Seifert disks for an oriented link to give a combinatorial algorithm to compute Massey's  higher order linking numbers.\end{abstract}

\section{Introduction} 

It is well-known that there are many different methods of computing the linking number of an oriented 2-component link in $S^3$ (see for example \cite[pp.132--135]{rolfsen}), and perhaps the simplest is the  combinatorial formula which counts the number of signed crossings of one component going under another component in a diagram of the link.  It is much more subtle to compute higher-order linking numbers, and it has been a folklore to use the intersection theory in the process, which was first suggested by W. Massey.  In \cite{massey 60} Massey introduced the higher-order linking as an application of his higher-order cohomology operations defined in terms of suitable cochains \cite{massey 58}, and he calculated the third-order linking numbers for some 3-component links by  shifting from cohomology and cup product to homology and intersection theory via duality theorems for manifolds.  Later several works along this direction (\cite{cochran},\cite{stein},  \cite{hain 85})  gave methods in various forms for computing the higher-order linking numbers, and others computed Milnor's $\bar \mu$ -invariants (e.g. \cite{mellor}) by using their well-known connection to Massey invariants (\cite{porter},\cite{turaev},\cite{fenn}); but a formal derivation of the general formulae in intersection theory which are used for the computation has not yet been given, and a concrete algorithm for the computation is hence lacking.  In this paper we will complete Massey's original approach by developing systematically  the intersection theory at the relative chain-cochain level in the simplicial category, and use it to derive recursive combinatorial formulae for computing all higher-order linking numbers.  The formulae are algorithmic in the sense that the computation of an $n$-th order linking requires the construction of certain surfaces from the assumption that the $(n-1)$-st order linking being 0.  From a diagram of a link, these formulae give rise to a combinatorial algorithm for computing higher-order linkings by using Seifert disks of the link to facilitate the construction of the intersection of Seifert surfaces of the link, which is necessary for the inductive step.    
 
\phantom{.}  

The paper is organized into 5 sections.  Section 2 gives preliminaries for the definition and results in section 3 on the intersection product at the relative chain-cochain level.  In section 4 explicit formulas for the Massey higher-order linking numbers in terms of intersection product of relative chains are presented, and a geometric topology interpretation for the relative chains and their intersection product is given.  A combinatorial algorithm for computing Massey third-order linking number is given in section 5, and to demonstrate its use we apply the algorithm on two examples.

\section{Preliminaries}
  
We work in the piecewise linear category.   Most material in this section is in the area of combinatorial topology and can be found in \cite{lefschetz}, \cite{hilton} and \cite{whitehead}.  Let $M$ be a closed oriented $n$-manifold and $N$ a closed submanifold of $M$, and let $(K,L)$ be a fixed triangulation of $(M,N)$ with a given ordering of the vertices (this means a partial ordering of the vertices of $K$ such that the vertices of each simplex are totally ordered).  Consider the barycentric subdivision $(K',L')$ of $(K,L)$, and for a $p$-simplex $\sigma$ of $K$, denote by $D(\sigma)$ the dual $(n-p)$-cell of $\sigma$ with orientation given by the transverse orientation of $\sigma$.  Explicitly, $D(\sigma)$ is the subcomplex of $K'$ whose vertices are barycenters of simplexes of $K$ having $\sigma$ as a face.  The collection $K^*=\{D(\sigma) |  \phantom{.} \sigma$ is a simplex of $K\}$ forms a cell decomposition (or block decomposition) of $K'$.  Note that $L^*=N(L')$ as sub-complexes of $K'$, the latter being the regular neighborhood of $L'$ in $K'$.  Clearly $L^*$ is a sub-cell-complex of $K^*$.  Let $C_q(K^*) ( (C_q(L^*)$ respectively) be the free abelian group generated by the $q$-cells $D(\sigma)$ of $K^*$ ($L^*$, respectively). The boundary operator $\partial : C_q(K^*) \to C_{q-1}(K^*)$ is defined by setting $\partial D(\sigma)$ to be the subcomplex of $K'$ whose vertices are barycenters of simplexes of $K$ having $\sigma$ as a proper face.  If $\sigma$ is a $p$-simplex of $K$, then $D(\sigma)$ is an $(n-p)$-cell in $K^*$, and $\partial D(\sigma)$ is a union of $D(\tau)$ for some $(p+1)$-simplex $\tau$ of $K$, so $\partial D(\sigma) \phantom{.}\epsilon \phantom{.}C_{n-p-1}(K^*)$ is a sum of $(n-p-1)$-cells in $K^*$.  If $a = \sum a_i \sigma_i \phantom{.}\epsilon \phantom{.} C_p(K)$, then let $D(a) = \sum a_i D(\sigma_i) \phantom{.}\epsilon \phantom{.} C_{n-p}(K^*)$.  We have the chain complex $\{C_q(K^*),\partial \}$ and the sub-chain complex $ \{C_q(L^*),\partial \}$, and the boundary operator $\partial$ induces the boundary operator on the relative chain complex $\overline \partial : C_q(K^*, L^*) \to C_{q-1}(K^*,L^*)$.  Recall that the subdivision operator $sd : (K,L) \rightarrow (K',L')$ induces a chain map $sd_\# :  C_p(K,L) \rightarrow C_p(K',L')$,  and a simplicial map $\theta : (K',L') \to (K,L)$ which gives rise to a chain homotopy inverse $\theta_\# : C_p(K',L') \to C_p(K,L)$ of $sd_\#$ is defined as follows.  For any vertex $b$ of $K'$, $b$ is the barycenter of a unique simplex $\sigma$ of $K$.  Define $\theta(b)$ to be an arbitrary chosen vertex of $\sigma$, and extend $\theta$ piecewise linearly over $K'$.  We shall use the convention that $\theta(b)$ is the last vertex in the ordering of vertices of $\sigma$.  It is clear that  $\theta_\# \circ sd_\# = id$ on $C_p(K,L)$, and dually we have the cochain maps on relative cochain complexes $sd^ \# :  C^p(K',L') \rightarrow C^p(K,L)$ and $\theta ^ \# : C^p(K,L) \to C^p(K',L')$,  and we have $  sd^\# \circ \theta^\# = id$ on $C^p(K,L)$.

Consider first the absolute case where $L = \phi$.  For a $p$-simplex $\sigma$ of $K$, let $u_\sigma \phantom{.} \epsilon \phantom{.} C^p(K)$ be the dual of $\sigma$ satisfying 

\begin{equation} \label{product}
u_\sigma(\tau) = \left\{   
\begin{array} {cl}
1 & \mathrm{if} \quad \tau = \sigma, \\
0 &\mathrm{otherwise}.
\end{array}
 \right.
 \end {equation}
for any simplex $\tau$ of $K$.  Then $  B = \{u_\sigma\phantom{.} \big | \phantom{.} \sigma$ is a $p$-simplex of $K\}$ is a basis for $C^p(K)$.  Let $\xi \phantom{.} \epsilon \phantom{.} Z_n(K)$ be the $n$-cycle representing the fundamental class of $K$.  It is known \cite{whitehead} that $sd_\#(\xi) \cap \theta^\#(u_\sigma) = D(\sigma)$ for any simplex $\sigma$ of $K$, here $D(\sigma)$ is considered as an element of $C_{n-p}(K')$, and the map $\phi : C^p(K) \to C_{n-p}(K^*)$ given by $\phi(u_\sigma)= D(\sigma)$ and extends linearly is an isomorphism satisfying $\phi(\delta \mu_\sigma)= (-1)^{p+1}\partial D(\sigma)$, from which follows the Poincare duality.   In \cite[p. 508]{whitehead} , the intersection product is defined  to be the pairing $: C_p(K) \otimes C_q(K^*) \overset  \bullet \to C_{p+q-n}(K')$ (with $p+q \ge n$) given by the composition 

\begin{eqnarray*}
C_p(K) \otimes C_q(K^*)& \overset {1\otimes \phi^{-1}}  \longrightarrow&  C_p(K) \otimes C^{n-q}(K) 
\overset {sd_\# \otimes 1} \longrightarrow C_p(K') \otimes C^{n-q}(K)\\ 
&\overset {1\otimes \theta^\#}  \longrightarrow& C_p(K') \otimes C^{n-q}(K') 
 \overset \cap \longrightarrow  C_{p+q-n}(K ').
\end{eqnarray*}

\phantom{.}

\emph{Remark}: It follows from the definition that for any $p$-simplex $\sigma$ and any $(n-q)$-simplex $\tau$ of $K$,  $\sigma \cdot D(\tau)= sd_\#(\sigma) \cap \theta^\#(u_\tau) $, and by the topological definition of the cap product (see, for example [1], page 204), that $\sigma \cdot D(\tau)$ is an $(n-p)$-simplex in $K'$ whose underlying set  $|\sigma \cdot D(\tau)|$ satisfies $|\sigma \cdot D(\tau)| \subset |sd_\#(\sigma)|$, and $|\sigma \cdot D(\tau)| = |sd_\#(\sigma) \cap \theta^\#(u_\tau)| \subset |sd_\#(\xi ) \cap \theta^\#(u_\tau)| = |D(\tau)|$, so $|\sigma \cdot D(\tau)|$ is contained in the intersection of $|sd_\#(\sigma)|$ and $|D(\tau)|$.  Conversely, given a simplex $\eta$ of $K'$ such that $|\eta| \subset  |sd_\#(\sigma)| \cap |D(\tau)|$, we have $|\eta|\phantom{.}\subset \phantom{.} |D(\tau)| = |sd_\#(\xi) \cap \theta^\#(u_\tau)| = |[ sd_\#(\xi)\lambda_p, \theta^\#(u_\tau)](sd_\#(\xi)\rho_{n-p})|$, and since  $|\eta| \phantom{.} \subset \phantom{.} |sd_\#(\sigma)|$, it follows that $|\eta| \phantom{.} \subset  \phantom{.} |[ sd_\#(\sigma)\lambda_p, \theta^\#(u_\tau)](sd_\#(\sigma)\rho_{n-p})| =  |sd_\#(\sigma) \cap \theta^\#(u_\tau) = |\sigma \cdot D(\tau)|$.  Thus  $|\sigma \cdot D(\tau)|$ is equal to the intersection of  $|sd_\#(\sigma)| = |\sigma|$ and $|D(\tau)|$, which justifies the "intersection" for the pairing defined.
  An example for the case $n=3, p=2,$ and $q=2$ is given in the Appendix, which will be used later for the combinatorial algorithm in section 5 to "see" the sense of direction of the intersection curve of two Seifert disks.
    
\phantom{.}

It follows from the fact (see, e.g. \cite{whitehead}) that for $ a \phantom{.} \epsilon \phantom{.} C_p(K)$ and $b \phantom{.} \epsilon \phantom{.} C_q(K^*)$, 

\begin{eqnarray*}
\partial ( a \cdot b ) = (-1)^{n-q} (\partial a) \cdot b + (-1)^{n+1} a \cdot (\partial b), 
\end{eqnarray*}
 there is an induced intersection product on the homology classes :  For $[a] \phantom{.} \epsilon \phantom{.} H_p(K)$ and $[b] \phantom{.} \epsilon \phantom{.} H_q(K^*)$, $[a] \cdot [b] = [ a \cdot b ]\phantom{.} \epsilon \phantom{.} H_{p+q-n}(K')$.
 
 \phantom{.}
 
 We will use the following fact of the functional property of cap product (see, e.g. \cite{greenberg})
 
 \begin{lemma}
 For any map $f : X \to Y$, and any $a \phantom{.}\epsilon \phantom{.}C_{p+q}(X)$ and $b \phantom{.}\epsilon \phantom{.}C^p(Y)$, 
 
 \begin{eqnarray*}
 f_\#(a \cap f^\#(b)) = f_\#(a) \cap b.
 \end{eqnarray*}  
 
The same formula holds for $f : (X,A) \to (Y,B)$, and any $a \phantom{.}\epsilon \phantom{.}C_{p+q}(X,A)$ and $b \phantom{.}\epsilon \phantom{.}C^p(Y,B)$.
\end{lemma}
 
\section{Intersection product of relative chains}

Given the simplicial complex pair $(K,L)$, let $N(L)$ be a regular neighborhood of $L$ in $K$.  Every $g \phantom{.} \epsilon \phantom{.} C^p(K- \overset \circ N(L))$ extends to a unique $\overset \thicksim g \phantom{.} \epsilon \phantom{.} C^p(K)$ satisfying $\overset \thicksim  g(\triangle)=0$ for any $p$-simplex $\triangle$ whose support is contained in $N(L)$.  So we may consider $C^p(K - \overset \circ N(L))$ as a sub-cochain complex of $C^p(K)$, of $C^p(K,N(L))$, and of $C^p(K,L)$.  Considering $C^p(K - \overset \circ N(L))$ as a sub-cochain complex of $C^p(K)$, we have

\phantom{.}

\begin{theorem} \label{GBF1}
The map $\bar \phi = \phi \Big |_{C^p(K - \overset \circ N(L))} : C^p(K - \overset \circ N(L)) \to C_{n-p}(K^*,L^*)$ given by $\bar \phi(u_\sigma) = \overline {D(\sigma)}$ for any $p$-simplex $\sigma$ of $K - \overset \circ N(L)$, is an isomorphism satisfying $\overline \phi (\delta u_\sigma) = (-1)^{p+1}\phantom{.}
\overline \partial \phantom{.} \overline {D(\sigma)} \phantom{.} (= (-1)^{p+1}\phantom{.}\overline {\partial D(\sigma)})$.
\end{theorem}

\begin{proof}  Since $\underset {\sigma} {\bigcup}D(\sigma)=K^* - L^*$, where the union ranges over all $p$-simplexes $\sigma$ of $K  - \overset \circ N(L)$, the collection $\{D(\sigma)| \phantom{.} \sigma $ is a $p$-simplex of $K - \overset \circ N(L) \}$ spans $C_{n-p}(K^* - L^*)$ freely.  Now $D(\sigma) = \phi(u_\sigma)$ and  $\phi : C^p(K) \to C_{n-p}(K^*)$ is an isomorphism, the restriction 

\begin{eqnarray*}
\bar \phi = \phi \Big |_{C^p(K - \overset \circ N(L))} : C^p(K - \overset \circ N(L)) \to C_{n-p}(K^* - L^*)
\end{eqnarray*} 
is an isomorphism.  Since 

\begin{eqnarray*}
C_{n-p}(K^*) = C_{n-p}(K^* - L^*) \oplus C_{n-p}(L^*),
\end{eqnarray*} 
the result follows from the natural identification 

\begin{eqnarray*}
C_{n-p}(K^* - L^*) \cong \displaystyle \frac {C_{n-p}(K^*)} {C_{n-p}(L^*)} = C_{n-p}(K^*, L^*),
\end{eqnarray*} 
in which $D(\sigma)$ in $C_{n-p}(K^* - L^*)$ is identified with $\overline {D(\sigma)}$ in $\displaystyle C_{n-p}(K^*, L^*)$.

\end{proof}

\phantom{.}

\emph{Remark}:  If $C_p(K - \overset \circ N(L))$ is considered as a sub-chain complex of $C_p(K,L)$, then for any $p$-simplex $\sigma$ of $K - \overset \circ N(L)$, $\bar \phi(\mu_{\sigma} )= sd_\#(\xi\Big |_{K - \overset \circ N(L)}) \cap \theta^\#(u_\sigma)$, where $\xi$ is the fundamental class of $K$.

\phantom{.}

Thus we have proved the following version of the Alexander duality:

\begin{corollary} \label{GBF1}The  isomorphism $\overline \phi$ induces an isomorphism (still denoted $\overline \phi$)

\begin{eqnarray*}
\overline \phi: H^p(K  - \overset \circ N(L)) \to H_{n-p}(K^*,L^*) .
\end{eqnarray*} 

\end{corollary}

\phantom{.}

The intersection product at the relative chain-cochain level is the pairing

\begin{eqnarray*}
C_p(K,L) \otimes C_q(K^*,L^*) \overset \bullet \to C_{p+q-n}(K',L')
\end{eqnarray*} 
(with $p+q \ge n$) given by the composition 

\begin{eqnarray*}
C_p(K,L) \otimes C_q(K^*,L^*) \overset {1\otimes \phi^{-1}} \longrightarrow C_p(K,L) \otimes C^{n-q}(K - \overset \circ N(L)) 
\overset {sd_\# \otimes 1} \longrightarrow C_p(K',L') \otimes C^{n-q}(K - \overset \circ N(L))\\ 
\overset {1\otimes \theta^\#} \longrightarrow C_p(K',L') \otimes C^{n-q}((K - \overset \circ N(L))') 
\overset {1\otimes i} \longrightarrow C_p(K',L') \otimes C^{n-q}(K ',L') \overset \cap \longrightarrow  C_{p+q-n}(K ',L'),
\end{eqnarray*}
where $i : C^{n-q}((K - \overset \circ N(L))')  \longrightarrow C^{n-q}(K ',L')$ is the naturally induced mapping.

\phantom{.}

\emph{Remark}:   Similar to the absolute case, for simplexes $\sigma$ and $\tau$ of $K$, the underlying set $|\overline \sigma \cdot \overline {D(\tau)}|$ is the intersection of $|sd_\#(\sigma)|= |\sigma|$ and $|D(\tau)|$ mod $|L'|= |L|$, and in general, for any $\bar a \phantom{.}\epsilon \phantom{.} C_p(K,L)$ and any $\bar b \phantom{.}\epsilon \phantom{.} C_q(K^*,L^*)$, $|\bar a \cdot \bar b|$ is the intersection of $|sd_\#(a)|$ and $|D(b)|$  mod $|L'|$.  Furthermore, $|\theta_\#(\overline \sigma \cdot \theta_\#{D(\tau)})|$ is the intersection of $ |\sigma|$ and $|\theta_\#(D(\tau))|$ mod $|L|$.

\phantom{.}

Correspondingly there is  an induced intersection product on relative homology classes :  For  $[a] \phantom{.} \epsilon \phantom{.} H_p(K,L)$ and $[b] \phantom{.} \epsilon \phantom{.} H_q(K^*,L^*)$, $[a] \cdot [b] = [ a \cdot b ] \phantom{.} \epsilon \phantom{.} H_{p+q-n}(K',L')$.

\phantom{.}
 
Recall that for a topological space $X$, the epimorphism $\partial_\# : C_0(X) \to Z$ defined by $\partial_\#( \sum m_i p_i) = \sum m_i$, where $p_i$ are points in $X$, satisfies the property $\partial_\# (c \cap \eta) = [c,\eta] \phantom{.} (=\eta(c))$ for any $c \phantom{.} \epsilon \phantom{.} C_p(X)$ and $\eta \phantom{.} \epsilon \phantom{.} C^p(X)$.  If $X$ is path connected, then $\partial_\# (B_0(X)) =0$.  One can similarly define 
$\partial_\# : C_0(X,A) \to Z$, which  satisfies $\partial_\# (\bar c \cap \bar \eta) = [\bar c,\bar \eta] $ for any $\bar c \phantom{.} \epsilon \phantom{.} C_p(X,A)$ and $\bar \eta \phantom{.} \epsilon \phantom{.} C^p(X,A)$.

\phantom{.}

It follows from the definition that for any $p$-simplex $\sigma$ of $K$ and $(n-q)$-simplex $\tau$ of $K - \overset \circ  N(L), \phantom{..}  \overline \sigma  \phantom{.}   \cdot   \phantom{.}   \overline {D(\tau)} = sd_\#(\sigma) \cap \theta^\#(u_\tau)$.  In general, for any $ \overline a \phantom{.} \epsilon \phantom{.} C_p(K,L)$ and $v \phantom{.} \epsilon \phantom{.} C^{n-q}(K - \overset \circ N(L))$, $ \overline a \cdot \overline \phi (v) = sd_\#(a) \cap \theta^\#(v)$.  In particular, if $p+q=n$, then $\partial_\#( \overline a \cdot \overline \phi(v) )= [sd_\#(a) , \theta^\#(v)]$, here $\theta^\#(v) \phantom{.} \epsilon 
\phantom{.} C^{p}((K - \overset \circ N(L))')$ is identified with $i(\theta^\#(v))$ in $C^{p}(K ',L')$, where $i : C^p((K - \overset \circ N(L))')   \to C^p(K',L')$ is the naturally induced map. 

\phantom{.}

\begin{proposition} \label{GBF2}
For $\overline a \phantom{.} \epsilon \phantom{.} C_{n}(K,L)$ and $\beta \phantom{.}\epsilon \phantom{.} C^p(K)$, $\overline a \cap \beta =  \theta_\# ( sd_\# (\overline a) \cap \theta^\#(\beta))$.
\end{proposition}

\begin{proof}  By Lemma 2.0.1,  we have

\begin{eqnarray*}
\theta_\#(sd_\#(\overline a) \cap \theta^\#(\beta)) &= &  \theta_\#({sd_\#(\overline a))} \cap \beta \\
&=&  \overline a \cap \beta,
\end{eqnarray*} 
since $\theta_\# \circ sd_\# = id$.
\end{proof}

\begin{proposition} \label{GBF2}
Let $\alpha \phantom{.} \epsilon \phantom{.} C^p(K)$, $\beta \phantom{.}\epsilon \phantom{.} C^q(K)$, and $T \phantom{.} \epsilon \phantom{.} C_{p+q}(K)$.  Then $(\alpha \cup \beta )(T) = \partial_\#(\theta_\#(\theta_\#(T \cdot \phi(\alpha)) \cdot \phi(\beta)))$.
\end{proposition}

\begin{proof}  The result is obtained from the following direct computation using Proposition 3.0.4 and the definition of intersection product.

\begin{eqnarray*}
(\alpha \cup \beta)(T) & = & \partial_\#(T \cap (\alpha \cup \beta)) = \partial_\#((T\cap \alpha) \cap \beta) \\
&= &  \partial_\# (\theta_\#(sd_\#(T \cap \alpha) \cap \theta^\#(\beta))\\
&=&  \partial_\#(\theta_\#((T \cap \alpha) \cdot (\phi(\beta))) \\
&=&  \partial_\#(\theta_\#(\theta_\#( sd_\#(T) \cap \theta^\#(\alpha)) \cdot \phi(\beta))) \\
&=& \partial_\#(\theta_\#(\theta_\#(T \cdot \phi(\alpha) ) \cdot \phi(\beta))).
\end{eqnarray*} 

\end{proof}

Note that the above result extends to the case where $\alpha \phantom{.} \epsilon \phantom{.} C^p(K - \overset \circ  N(L))$, $\beta \phantom{.}\epsilon \phantom{.} C^q(K - \overset \circ  N(L))$, and $T \phantom{.} \epsilon \phantom{.} C_{p+q}(K - \overset \circ  N(L))$, which is the case we will consider later.  In this case, for the right hand side $T$ may be considered as in $C_{p+q}(K,L)$, and $\alpha$ as in $C^p(K)$, so $T \cap \alpha \phantom{.} \epsilon \phantom{.} C_q(K,L)$ and the intersection product $(T \cap \alpha) \cdot \overline \phi(\beta)$ is defined.

\phantom{.}

The deformation retract from $(|K'|,|N(L')|)$ to $(|K'|,|L'|)$ of the underlying spaces is homotopically equivalent to a simplicial map $r : (|K'|,|N(L')|)\to(|K'|,|L'|)$ via simplicial approximation of the deformation retract, and it induces an isomorphism $r_* : H_q(K',N(L')) \to H_q(K',L')$ as simplicial homology groups ( \cite[Theorem 3.6.6, p.120]{hilton}), which is the inverse of the inclusion-induced isomorphism $j_* : H_q(K',L') \to H_q(K',N(L'))$.

\begin{lemma} \label{GBF2}
The inclusion map $i : (K^*,L^*) \to (K',N(L'))$ induces an isomorphism $i_* : H_q(K^*,L^*) \to H_q(K',N(L'))$.
\end{lemma}

\begin{proof}  By \cite[Theorem 3.8.8, p. 131]{hilton}, the inclusion map $i : K ^* \to K'$ induces an isomorphism $i_* : H_q(K^*) \to H_q(K')$.  Now since $L^*$ is a block dissection of $N(L')$, the inclusion map $i : L^* \to N(L')$ induces an isomorphism $i_* : H_q(L^*) \to H_q(N(L'))$.  The result follows from the long exact sequences of $H_*(K^*,L^*)$ and  $H_*(K',N(L'))$, and by applying the Five Lemma:

$$
\begin{array}{ccccc}
H_q(L^*) & \longrightarrow  H_q(K ^*)  & \longrightarrow H_q(K ^*,L^*)  & \longrightarrow H_{q-1}(L^*)  & \longrightarrow  H_{q-1}(K ^*) \\
\phantom{..}\Bigg \downarrow { i_*} & \phantom{.....}\Bigg \downarrow { i_*} & \phantom{.....}\Bigg \downarrow { i_*}  & \phantom{.....}\Bigg \downarrow { i_*} & \phantom{.....}\Bigg \downarrow { i_*}   \\
\ H_q(N(L'))  & \longrightarrow  H_q(K ') & \longrightarrow H_q(K',N(L')) & \longrightarrow H_{q-1}(N(L')) &  \longrightarrow  H_{q-1}(K') \\ 
\end{array}
$$

\end{proof}  

We now restrict to the case where $p=1$, $q=2$, and $n=3$, and consider the following diagram

 $$
 \begin{array}{ccccc}
 H^1(K - \overset \circ N(L) )  \otimes  H^1(K - \overset \circ N(L) ) && \buildrel \cup \over \longrightarrow && H^2(K - \overset \circ N(L) )  \\
\phantom{.........}\Bigg \downarrow {\overline \phi \otimes \overline \phi}   &&&& \Bigg \downarrow {\overline \phi}\\
 H_2(K^*,L^*) \otimes  H_2(K^*,L^*) &&   && H_1(K^*,L^*) \\
 \phantom{...................}\Bigg \downarrow {(\theta_* \circ r_* \circ {i}_*) \otimes id } &&&& \phantom{.......}\Bigg \downarrow \ { {r_* \circ i_*}}  \\
 H_2(K,L) \otimes H_2(K^*,L^*) && \buildrel \bullet \over\longrightarrow &&    H_1(K',L')\\
 \end{array}
 $$
 
 \phantom{.}
 
 \begin{proposition} \label{GBF2}
The above diagram commutes. i.e. for any $[\alpha]$, $[\beta] \phantom{.} \epsilon \phantom{.} H^1(K - \overset \circ N(L))$, $( r_* \circ i_*)(\big [ \phantom{.}{\overline \phi(\alpha \cup \beta)}\phantom{.} \big ] ) = ((\theta_*  \circ r_*  \circ i_*) \big ([ \phantom{.}\overline \phi(\alpha)\phantom{.} \big ] )) \cdot \big[\overline \phi(\beta) \big ] $
\end{proposition}

\begin{proof} 

For any $[\alpha]$, $[\beta]$ in $H^1(K - \overset \circ N(L))$, 

\begin{eqnarray*}
(( r_* \phantom{.}\circ \phantom{.}  i_*) \circ \overline {\phi}) ([\alpha] \cup [\beta] ) &=& ( r_* \circ i_*)(\overline {\phi} ( [\alpha \cup \beta]))\\
 & = & ( r_*\circ i_*)(\big [ \overline {sd_\#(\xi) \cap \theta^\#(\alpha \cup \beta)}\big ] ) \\
 &=& \big [ \overline {sd_\#(\xi) \cap \theta^\#(\alpha \cup \beta)}\big ],
 \end{eqnarray*} 
Considering $\alpha$ and $\beta$ as in $Z^1(K,L)$, then $\theta^\#(\alpha \cup \beta)\in Z^1(K',L')$, and we may consider $\big [ \overline {sd_\#(\xi) \cap \theta^\#(\alpha \cup \beta)}\big ]$ as in $H_1(K',L')$.  Then $i_*(\big [ \overline {sd_\#(\xi) \cap \theta^\#(\alpha \cup \beta)}\big ] )= j_*(\big [ \overline {sd_\#(\xi) \cap \theta^\#(\alpha \cup \beta)}\big ] )$, or equivalently $( r_*\circ i_*)(\big [ \overline {sd_\#(\xi) \cap \theta^\#(\alpha \cup \beta)}\big ] ) = \big [ \overline {sd_\#(\xi) \cap \theta^\#(\alpha \cup \beta)}\big ]$, since $j_* = r_*^{-1}$.

\phantom{.}

Now  as an element in $H_1(K',L')$,\phantom{..}
$\big [ \overline {sd_\#(\xi) \cap \theta^\#(\alpha \cup \beta)}\big ] = \big [ \overline {sd_\#(\xi) \cap (\theta^\#(\alpha) \cup \theta^\#(\beta))} \big ]  = \big[ \overline {(sd_\#(\xi) \cap \theta^\#(\alpha)) \cap \theta_\#(\beta)} \big]$.

\phantom{.}
 
Using the property $sd^\# \circ \theta^\# = id$ and by lemma 2.0.1 we obtain 

\begin{eqnarray*} 
sd_\#(\xi )\cap \theta^\#(\alpha) =  sd_\#(\xi \cap (sd^\# \circ \theta^\#)(\alpha)) = sd_\#(\xi \cap \alpha), 
\end{eqnarray*} 

so we have
 
\begin{eqnarray*} 
\big[ \overline {(sd_\#(\xi) \cap \theta^\#(\alpha)) \cap \theta_\#(\beta)} \big] = \big[ \overline {sd_\#(\xi \cap \alpha) \cap \theta^\#(\beta)} \big] \newline =  \big[ \overline {(\xi \cap \alpha) \cdot  (sd_\#(\xi )\cap \theta^\#(\beta))} \big] 
\end{eqnarray*}

\phantom{.} 
 
 It follows from Lemma 2.0.1 and the fact that $\theta_\# \circ sd_\# = id$ that 
 
\begin{eqnarray*} 
\big[ \overline {(\xi \cap \alpha) \cdot  (sd_\#(\xi )\cap \theta^\#(\beta))} \big]
&=& \big [\overline {(\theta_\# (sd_\#(\xi) \cap \theta^\#(\alpha) ) \cdot (sd_\#(\xi) \cap \theta^\#(\beta)}) \big] \\
&=& \theta_* (\big[ \overline {sd_\#(\xi) \cap \theta^\#(\alpha)} \big]) \cdot \big [  \overline {sd_\#(\xi) \cap \theta^\#(\beta) }\big] 
\end{eqnarray*} 
Note that here $\overline {sd_\#(\xi) \cap \theta^\#(\alpha)}$ is considered as in $C_2(K',L')$.  The last homology class is equal to

\begin{eqnarray*} 
((\theta_* \circ r_* \circ i_*)( \big [ \overline {sd_\#(\xi) \cap \theta^\#(\alpha)} \big]) )\cdot \big[ \overline {sd_\#(\xi) \cap \theta^\#(\beta)} \big] 
&=& ((\theta_* \circ r_* \circ i_*) \big[ \overline \phi (\alpha)]) \cdot \big[ \overline \phi(\beta) \big]\\
&=& (\theta_* \circ r_* \circ i_* \otimes id) ([\bar \phi(\alpha)], [\bar \phi(\beta)]).
\end{eqnarray*} 

\phantom{.}

\end{proof}

\vspace{0.2in}

\section{Linkings in $S^3$}

Let $K = S^3$ with the usual orientation and a fixed triangulation, $L =  \coprod K_i$ an oriented link in $S^3$, $T_i = \partial N(K_i)$, and $\mu_i \phantom{.}\epsilon \phantom{.}  Z_1(S^3- \overset \circ N(L))$ be the 1-cycle representing the meridian of $K_i$ (with $lk(K_i,\mu_i)=1$).  For each $i$, let $[\mu_i^*] \phantom{.}\epsilon \phantom{.}  H^1(S^3 - \overset \circ N(L)) = Hom( H_1(S^3 - \overset \circ N(L),Z)$ be the dual of $[\mu_i] \phantom{.}\epsilon \phantom{.}  H_1(S^3 - \overset \circ N(L))$  and let $F_i$ be a spanning surface of $K_i$ in $S^3$.  It is well known that the isomorphism 

\begin{eqnarray*}
H^1(S^3  - \overset \circ {N(L)}) \to H_2(S^3,L) 
 \end{eqnarray*} 
given by the cap product with the fundamental class $[S^3  - \overset \circ N(L)]$, takes $[\mu_i^*] $ to $[F_i]$, that is, $[S^3  - \overset \circ N(L)] \cap[\mu_i^*] = [F_i]$.   On the other hand, the Alexander duality in Corollary 3.0.3 

\begin{eqnarray*}
\overline \phi: H^1(S^3  - \overset \circ N(L)) \to H_2((S^3)^*,L^*),
\end{eqnarray*} 
gives 

\begin{eqnarray*}
\overline \phi ([\mu_i^*]) &=&[sd_\#(S^3  - \overset \circ N(L) ) \cap \theta^\#(\mu_i^*)] \\
&=& sd_*([S^3  - \overset \circ N(L)]) \cap \theta^*([\mu_i^*]) \\
&=& \theta_*^{-1} ([S^3  - \overset \circ N(L)])  \cap \theta^*([\mu_i^*]).
\end{eqnarray*} 
By Lemma 2.0.1 again, the last expression is equal to $\theta_*^{-1} ([S^3  - \overset \circ N(L)]  \cap {\theta^*}^{-1}(\theta^*([\mu_i^*]))$, which is equal to ${\theta_*}^{-1} ([S^3  - \overset \circ N(L)]  \cap [\mu_i^*] )= \theta_*^{-1}([F_i])$, and is equal to $i_*^{-1} \circ j_* \circ \theta _*^{-1} ([F_i])$ as an element in $H_2((S^3)^*,L^*)$.  This shows that $[F_i] = (\theta_* \circ r_* \circ i_*)(\overline \phi ([\mu_i^*]))$.  It follows that  $(\theta_\# \circ r_\# \circ i_\#)(\overline \phi ([\mu_i^*])) = F_i + b_i$ for some $b_i \in  B_2(S^3,L)$.  So the underlying set $|F_i| = |(\theta_\# \circ r_\# \circ i_\#)(\overline \phi ([\mu_i^*]))| \cup |b_i|$.

\phantom{.}

As an example we apply the intersection theory developed in the previous section to obtain a formula for the ordinary (second-order) linking number of a link with two components.

\phantom{.}

\emph{Example.} Let $L = K_1 \cup K_2$ be an oriented link of two components in the given specific order with meridians $\mu_1, \mu_2$ respectively.  Then $lk(K_1,K_2) = ([\mu_1^*] \cup [\mu_2^*]) ([T_1])(= ([\mu_2^*] \cup [\mu_1^*]) ([T_2])= -([\mu_2^*] \cup [\mu_1^*]) ([T_1]))$.  By Proposition 3.0.5, we have

\begin{eqnarray*}
(\mu_1^* \cup \mu_2^*)(T_1)= \partial_ \sharp (\theta_\sharp(\theta_\sharp(T_1 \cdot \overline \phi(\mu_1^*) ) \cdot \overline \phi(\mu_2^*) ).
\end{eqnarray*}
We first note that since $|T_1|\subset S^3  - \overset \circ N(L)$, $|\partial_ \sharp (\theta_\sharp(\theta_\sharp(T_1 \cdot \overline \phi(\mu_1^*) )))|= |T_1| \cap |(\theta_\# \circ r_\# \circ i_\#)(\overline \phi ([\mu_i^*]))| \subset |T_1|.$  It follows that

\begin{eqnarray*}
|\theta_\sharp(\theta_\sharp(T_1 \cdot \overline \phi(\mu_1^*) ) \cdot \overline \phi(\mu_2^*) )|
&=& | \theta_\sharp(T_1 \cdot \overline \phi(\mu_1^*) )| \cap |(\theta_\# \circ r_\# \circ \circ i_\#)(\overline \phi(\mu_2^*) )|\\
&=& |T_1| \cap | F_1 +b_1| \cap |F_2 + b_2| \\
&=& |T_1| \cap (| F_1| \cup |b_1|) \cap (|F_2| \cup |b_2|) \\
&=& (|T_1| \cap | F_1| \cap |F_2|) \cup (|T_1| \cap | F_1| \cap |b_2|) \\
&\phantom{....}& \phantom{...}\cup (|T_1| \cap |b_1| \cap |F_2|) \cup (|T_1| \cap | b_1| \cap |b_2|).\\
\end{eqnarray*}
Now $b_i=\partial \beta_i+\overset  \thicksim {b_i}$, $\beta_i \phantom{.}\epsilon \phantom{.} C_3(S^3)$ and $\overset  \thicksim {b_i} \phantom{.} \epsilon \phantom{.} C_2(L)$ for $i=1,2$, the second part in the above union can be written as
$|T_1| \cap | F_1| \cap |b_2| = (|T_1| \cap | F_1| \cap |\partial \beta_2|) \cup (|T_1| \cap | F_1| \cap |\overset  \thicksim {b_2}|$.

\phantom{.}

Since $|T_1| \cap | F_1|$ is a circle in $T_1$ parallel to $K_1$, and $|\partial \beta_2|$ is homeomorphic to a closed surface in $S^3$, the algebraic intersection number of $|T_1| \cap | F_1|$ with $|\partial \beta_2|$ is $0$.  Clearly $|T_1| \cap | F_1| \cap |\overset  \thicksim {b_2}| = \phi$(the empty set).  So the algebraic intersection number due to the second part in the union is $0$.  Similar arguments apply to show that the algebraic intersection number of the third and the fourth parts in the union are both $0$.
Thus, using $\#$ for the algebraic intersection number, we have 

\begin{eqnarray*}
(\mu_1^* \cup \mu_2^*)(T_1)
&= &\partial_ \sharp (\theta_\sharp(\theta_\sharp(T_1 \cdot \overline \phi(\mu_1^*) ) \cdot \overline \phi(\mu_2^*) )\\
&=& \# ( |T_1| \cap | F_1| \cap |F_2|)\\
&=& \# ( |K_1| \cap |F_2|)\\
&=& lk(K_1,K_2).
\end{eqnarray*}

\emph{Remark}:   One can also show that 
\begin{eqnarray*}
(\mu_2^* \cup \mu_1^*)(T_2)= \partial_ \sharp (\theta_\sharp(\theta_\sharp(T_2 \cdot F_2 ) \cdot F_1 )) = lk(K_1,K_2).
\end{eqnarray*}

Now we consider the third-order linking.  Let $L = K_1 \cup K_2 \cup K_3$ be an oriented link of three components in the given specific order with meridians $\mu_1, \mu_2,\mu_3$ respectively.  Then 

\begin{eqnarray*}
H^1(S^3 - \overset \circ N(L)) = < [\mu_1^*], [\mu_2^*], [\mu_3^*]> \cong Z^3.
\end{eqnarray*} 

Assume $lk(K_1,K_2) = lk(K_2,K_3)=lk(K_1,K_3) = 0$.  Now since
$H_2(S^3 - \overset \circ N(K_1 \cup K_2)) = <[T_1]> (=<[T_2]>) \cong Z$, and $\displaystyle ([\mu_1^*] \cup [\mu_2^*])([T_1]) = (\mu_1^* \cup\mu_2^*)(T_1) = lk(K_1,K_2) = 0 $, which implies that $[\mu_1^* \cup\mu_2^*] = 0$ in $H^2(S^3 - \overset \circ N(K_1 \cup K_2))$, which in turn implies that $[\mu_1^* \cup\mu_2^*] = 0$ in $H^2(S^3 - \overset \circ N(L))$.  Thus we have $\mu_1^* \cup\mu_2^* = \delta \overset  \thicksim  C_{12}$ for some  $\overset  \thicksim  C_{12} \phantom{.} \epsilon \phantom{.} C^1(S^3 - \overset \circ N(L))$.  Similarly $\mu_2^* \cup\mu_3^* = \delta \overset  \thicksim  C_{23}$ for some  $\overset  \thicksim  C_{23} \phantom{.} \epsilon \phantom{.} C^1(S^3 - \overset \circ N(L))$.  Then the Massey third-order product of $L$, in the given order of components $K_1$, $K_2$, and $K_3$, is uniquely defined and is given by 

\begin{eqnarray*}
< [\mu_1^*], [\mu_2^*], [\mu_3^*]> = [\mu_1^* \cup \overset  \thicksim  C_{23} + \overset  \thicksim  C_{12} \cup \mu_3^*] \phantom{.} \in \phantom{.} H^2(S^3 - \overset \circ N(L)),
\end{eqnarray*} 
and by definition, Massey's third-order linking number of $L$, with its components in this order,  is 

\begin{eqnarray*}
< [\mu_1^*], [\mu_2^*], [\mu_3^*]> ([T_1]) =  [\mu_1^* \cup \overset  \thicksim  C_{23} + \overset  \thicksim  C_{12} \cup \mu_3^*] ([T_1]).
\end{eqnarray*}

\begin{theorem} \label{GBF2}
Massey's third order linking number is given by 

\begin{eqnarray*}
[\mu_1^* \cup \overset  \thicksim  C_{23} + \overset  \thicksim  C_{12} \cup \mu_3^*] ([T_1]) 
&=& \partial_\#(\theta_\#(\theta_\#(T_1 \cdot \overline \phi(\mu_1^*)) \cdot C_{23}^*) )+ \partial_\#(\theta_\#(\theta_\#(T_1 \cdot C_{12}^*) \cdot \overline \phi(\mu_3^*)))\\
&=& \# \Big( |T_1| \cap |F_1| \cap|C_{23}| \Big) + \# \Big(|T_1| \cap |C_{12}| \cap |F_3| \Big),  
\end{eqnarray*}
where $C_{12}^*= \overline \phi (\overset  \thicksim  C_{12})$, $C_{23}^*= \overline \phi (\overset  \thicksim  C_{23})$, and $C_{12} = (\theta \circ r \circ i)_{\#}(C_{12}^*)$, $C_{23} = (\theta \circ r \circ i)_{\#}(C_{23}^*)$.
 
\end{theorem}

\begin{proof} 

By Proposition 3.0.5 and its remark, we obtain 

\begin{eqnarray*}
[\mu_1^* \cup \overset  \thicksim  C_{23} + \overset  \thicksim  C_{12} \cup \mu_3^*] ([T_1]) 
&=& ( \mu_1^* \cup \overset  \thicksim  C_{23} + \overset  \thicksim  C_{12} \cup \mu_3^* )(T_1) \\
&=& (\mu_1^* \cup \overset  \thicksim  C_{23}) (T_1)+ (\overset  \thicksim  C_{12} \cup \mu_3^*)(T_1) \\
&= & \partial_\#(\theta_\#(\theta_\#(T_1 \cdot \overline \phi(\mu_1^*)) \cdot \overline \phi ( \overset  \thicksim  C_{23}))) + \partial_\#(\theta_\#(\theta_\#(T_1 \cdot \overline \phi (\overset  \thicksim  C_{12}))) \cdot \overline \phi ( \mu_3^*))\\
&=&\partial_\#(\theta_\#(\theta_\#(T_1 \cdot \overline \phi(\mu_1^*)) \cdot C_{23}^*) )+ \partial_\#(\theta_\#(\theta_\#(T_1 \cdot C_{12}^*) \cdot \overline \phi(\mu_3^*))).
\end{eqnarray*} 

Now the underlying set 

\begin{eqnarray*}
|\theta_\#(\theta_\#(T_1 \cdot \overline \phi(\mu_1^*)) \cdot C_{23}^*)|
&=& | \theta_\#(T_1 \cdot \overline \phi(\mu_1^*))| \cap |(\theta_\# \circ r_\# \circ i_\#)(C_{23}^*)|\\
&=&   |T_1| \cap (|F_1| \cup |b_1|) \cap|C_{23}| \\
&=&  \Big( |T_1| \cap |F_1| \cap|C_{23}| \Big) \cup \Big(  |T_1| \cap |b_1| \cap|C_{23}| \Big), 
\end{eqnarray*} 
and similarly $|\theta_\#(\theta_\#(T_1 \cdot C_{23}^*) \cdot \overline \phi(\mu_1^*))| = \Big( |T_1| \cap |C_{12}| \cap (|F_3|  \cup |b_3| ) \Big)$, where $b_1$ and $b_3$ are as in the above Example.
since $b_1 = \partial \beta_1 + \overset  \thicksim {b_1}$, where $\beta_1 \phantom{.} \epsilon \phantom{.} C_3(S^3)$ and $ \overset  \thicksim {b_1} \phantom{.} \epsilon \phantom{.} C_2(L)$,
$ |T_1| \cap |b_1| \cap|C_{23}| =  (|T_1| \cap |\partial \beta_1| \cap|C_{23}|) \cup  (|T_1| \cap |\overset  \thicksim {b_1}| \cap|C_{23}|)$.  Now $|T_1| \cap |\partial \beta_1| $ consists of disjoint curves that bound in $T_1$, and $C_{23} \phantom{.}\epsilon \phantom{.} C_2(S^3,L)$, it follows that $\# (|T_1| \cap |\partial \beta_1| \cap|C_{23}|) = 0$.  Also $|T_1| \cap |\overset  \thicksim {b_1}| \cap|C_{23}| = \phi$(the empty set), since $|T_1| \cap |\overset  \thicksim {b_1}| = \phi$.  Thus we have proved $\partial_\#(\theta_\#(\theta_\#(T_1 \cdot \overline \phi(\mu_1^*)) \cdot C_{23}^*) ) = \# \Big( |T_1| \cap |F_1| \cap|C_{23}| \Big)$.  A similar arguments can be applied to show that $\partial_\#(\theta_\#(\theta_\#(T_1 \cdot C_{12}^*) \cdot \overline \phi(\mu_3^*))) =  \# \Big(|T_1| \cap |C_{12}| \cap |F_3| \Big)$.

\phantom{.}

\end{proof}

\emph{Remark.}  It is clear that the third order linking number is independent of the choice of  $b_i \phantom{.} \epsilon \phantom{.} B_2(S^3,L)$ with which $(\theta \circ r \circ i)_{\#}(\overline {\phi}(\mu_i^*)) = F_i + b_i$. 
The same is true for the choices of  $\overset  \thicksim  C_{12}$ and  $\overset  \thicksim  C_{23}$ in $C^1( S^3 - \overset \circ {N(L)})$, which can be verified as follows.  Assume $\overset  \thicksim  C_{12}'$ and $\overset  \thicksim  C_{23}'$ are another choices with which $\mu_1^* \cup \mu_2^* = \delta \overset  \thicksim  C_{12}'$ and $\mu_2^* \cup \mu_3^* = \delta \overset  \thicksim  C_{23}'$.  For the first case we show that if $C_{12}'^* = \overline \phi ( \overset  \thicksim  C_{12}')$, then

\begin{eqnarray*}
\partial_\# (\theta_\#(\theta_\#(T_1 \cdot C_{12}^*) \cdot  \overline {\phi}(\mu_3^*))) = \partial_\#(\theta_\#(\theta_\#(T_1 \cdot C_{12}'^*) \cdot  \overline {\phi}(\mu_3^*))).
\end{eqnarray*} 
Now $\delta  (\overset  \thicksim  C_{12} - \overset  \thicksim  C_{12}' )= 0$ in $C^2(S^3 - \overset \circ N(L))$  implies  $\partial (C_{12}^* - C_{12}'^*) = 0$ in  \phantom{.} $C_1((S^3)^*,L^*)$.  Thus

\begin{eqnarray*}
\partial (\theta _\#( T_1 \cdot (C_{12}^* - C_{12}'^*))) &=& \theta _\#( \partial ( T_1 \cdot (C_{12}^* - C_{12}'^*))\\
&=& \theta _\# \big( - \partial T_1 \cdot (C_{12}^* - C_{12}'^*) +  T_1 \cdot 
(\partial (C_{12}^* - C_{12}'^*)) \big)\\ 
&=&  0.
\end{eqnarray*}
This shows that $\theta _\#( T_1 \cdot (C_{12}^*)$ and $\theta _\#( T_1 \cdot (C_{12}'^*)$ are each a collection of the same arcs in $C_1(S^3,L)$, i.e. same arcs in $S^3$ rel $L$.  But $T_1 = \partial N(K_1)$, so they are actually a collection of the same arcs in $S^3$.
We then have

\begin{eqnarray*}
\partial_\# \big (\theta_\#(\theta_\#(T_1 \cdot (C_{12}^* - C_{12}'^*)) \cdot  \overline {\phi}(\mu_3^*)) \big) &=& \theta_\# \big( \partial( \theta _\#( T_1 \cdot (C_{12}^* - C_{12}'^*)) ) \cdot \overline {\phi}(\mu_3^*) \big)  \\
&=&  0.
\end{eqnarray*}

\phantom{.}

Similarly we can show $\partial_\#(\theta_\#(\theta_\#(T_1 \cdot \overline {\phi}(\mu_1^*)) \cdot C_{23}^*) = \partial_\#(\theta_\#(\theta_\#(T_1 \cdot \overline {\phi}(\mu_1^*)) \cdot C_{23}'^*)$. 

\phantom{.}

The geometric topology meaning of $C_{12}$ and $C_{23}$ are as follows.  By Proposition 3.0.7 we have, on one hand, 

\begin{eqnarray*}
( r_* \circ i_*) ([ \overline \phi (\mu_1^* \cup \mu_2^*)] = \theta_*(r_*(i_*([ F_1]))) \cdot [F_2] = [\theta _\#(r_\#(i_\#(F_1))) \cdot F_2].
\end{eqnarray*} 
On the other hand,

\begin{eqnarray*}
( r_* \circ i_*)([ \overline \phi (\mu_1^* \cup \mu_2)] &=& ( r_* \circ i_*)([\overline \phi (\delta  \overset  \thicksim C_{12})]) \\ 
&=& ( r_* \circ i_*)([\partial \overline \phi (\overset  \thicksim C_{12})])\\ 
&=& ( r_* \circ i_*)([\partial C_{12}^*])\\
&=& [\partial  (( r \circ i)_\#(C_{12}^*))].
\end{eqnarray*}  
Thus we have $\partial ( ( r \circ i)_\#(C_{12}^*) )= \theta _\#(r_\#(i_\#(F_1))) \cdot F_2 \phantom{.}+ \phantom{.} b_{12}' $, where $ b_{12}' \phantom{.} \epsilon \phantom{.} B_1((S^3)', L')$.  Applying $\theta_\#$ on both sides we obtain $\partial ( ( \theta \circ r \circ i)_\#(C_{12}^*) )= \theta _\#((\theta_\# ( r_\#(i_\#(F_1))) \cdot F_2 \phantom{.}+ \phantom{.} b_{12} $, where $ b_{12} \phantom{.} \epsilon \phantom{.} B_1(S^3, L)$.  Whence
$b_{12} = \partial \beta_{12} + \overset  \thicksim b_{12}$,  with $\beta_{12} \phantom{.} \epsilon \phantom{.} C_2(S^3)$ and $\overset  \thicksim b_{12} \phantom{.} \epsilon \phantom{.} C_1(L)$.  Now since ($ \theta \circ r \circ i)_\#$ is the "retraction" from $((S^3)^*, L^*)$ to $(S^3,L)$, $C_{12}$ can be viewed as a 2-complex in $(S^3,L)$ whose boundary, after adding to the boundary of some 2-complex in $(S^3,L)$, is the intersection of $F_1$(which is identified with $(\theta \circ r \circ i)_\#(F_1)$) and $F_2$, rel $L$.  One such candidate for $C_{12}$ is any spanning surface for the new link components formed by the arcs and circles of $F_1 \cdot F_2$ and portions of components of $L$.  Note that if $C_{12}'$ is another choice of such spanning surface, then $\partial C_{12}' = \partial C_{12}$, so $C_{12}' - C_{12} = \bar \partial \bar b$ for some $\bar b \phantom{.}\epsilon\phantom{.} C_3((S^3)^*, L^*)$.  By letting $\overset  \thicksim C_{12}' = \bar \phi ^{-1} (C_{12}' )$ and $v=\bar \phi ^{-1} (\bar b) \phantom{.}\epsilon\phantom{.} C^0(S^3 - \overset \circ N(L))$, then $\overset \thicksim C_{12}' - \overset \thicksim C_{12} = \bar \phi ^{-1}(C_{12}' - C_{12}) = \bar \phi ^{-1}(\bar \partial \bar b)) =(\bar \phi ^{-1}\bar \partial ) (\bar \phi(v)) = \delta v$.  Thus the third-order linking is independent of the choice of $C_{12}$.  Similar choice can be made for $C_{23}$, which is any spanning surface for the new link components formed by arcs and circles of the intersections of $F_2$ and  $F_3$ rel $L$.  Therefore  $C_{12}$ and $C_{23}$ are constructed from the Seifert surfaces $F_1, F_2,F_3$,  and their boundaries $L$, and in general they are each a collection of surfaces.

\phantom{.} 

\emph{Remark.}  The first term in Theorem 4.0.8 can be interpreted as $lk(K_1, \partial C_{23})$.

\vspace{0.3in}

Next we consider the fourth-order linking.  Let $L = K_1 \cup K_2 \cup K_3 \cup K_4$ be an oriented link of four components in the given specific order with meridians $\mu_1, \mu_2,\mu_3, \mu_4$ respectively.  Assume all the second-order and  all third-order Massey products vanish, i.e. $\mu_i^* \cup \mu_j^*= 0$ and $<\mu_i^*,\mu_j^*,\mu_k^*> = 0$ for  all permutations $(i,j)$ and $(i,j,k)$ of $\{1,2,3,4\}$.  Then in particular we have $\mu_1^* \cup\mu_2^* = \delta \overset  \thicksim  C_{12}$, $\mu_2^* \cup\mu_3^* = \delta \overset  \thicksim  C_{23}$, and $\mu_3^* \cup\mu_4^* = \delta \overset  \thicksim  C_{34}$ for some  $\overset  \thicksim  C_{12} , \overset  \thicksim  C_{23},$ and $\overset  \thicksim  C_{34} \phantom{.} \epsilon \phantom{.} C^1(S^3 - \overset \circ N(L))$; also $<\mu_1^*,\mu_2^*,\mu_3^*> = 0$ implies $\mu_1^* \cup \overset  \thicksim  C_{23} + \overset  \thicksim  C_{12} \cup \mu_3^* = \delta \overset  \thicksim  C_{123}$
and $<\mu_2^*,\mu_3^*,\mu_4^*> = 0$ implies   $\mu_2^* \cup \overset  \thicksim  C_{34} + \overset  \thicksim  C_{23} \cup \mu_4^* = \delta \overset  \thicksim  C_{234}$
for some  $\overset  \thicksim  C_{123}$ and  $\overset  \thicksim  C_{234}\phantom{.} \epsilon \phantom{.} C^1(S^3 - \overset \circ N(L))$.  Then the Massey fourth-order product of $L$, in this specific order of components $K_1$, $K_2$, $K_3$ and $K_4$, is uniquely defined and is given by 

\begin{eqnarray*}
< [\mu_1^*], [\mu_2^*], [\mu_3^*], [\mu_4^*]> = [\mu_1^* \cup \overset  \thicksim  C_{234} + \overset  \thicksim  C_{12} \cup \overset  \thicksim  C_{34} + \overset  \thicksim  C_{123} \cup \mu_4^*] \phantom{.} \epsilon \phantom{.} H^1(S^3 - \overset \circ N(L)),
\end{eqnarray*} 
and by definition, Massey's fourth-order linking number of $L$, with its components in this order,  is 

\begin{eqnarray*}
< [\mu_1^*], [\mu_2^*], [\mu_3^*], [\mu_4^*]> ([T_1]) =  [\mu_1^* \cup \overset  \thicksim  C_{234} + \overset  \thicksim  C_{12} \cup \overset  \thicksim  C_{34} + \overset  \thicksim  C_{123} \cup \mu_4^*] ([T_1]).
\end{eqnarray*} 

By a similar discussion as in the case of third-order linking, we obtain

\begin{theorem} \label{GBF2}
Massey's fourth-order linking number is given by 

\begin{eqnarray*}
[\mu_1^* \cup \overset  \thicksim  C_{234} + \overset  \thicksim  C_{12} \cup \overset  \thicksim  C_{34} + \overset  \thicksim  C_{123} \cup \mu_4^*] ([T_1]) &=& \partial_\#(\theta_\#(\theta_\#(T_1 \cdot \bar \phi(\mu_1^*)) \cdot C_{234}^*) \\
&\phantom{....}& \phantom{.}+\partial_\#(\theta_\#(\theta_\#(T_1 \cdot C_{12}^*) \cdot C_{34}^*)) \\
&\phantom{....}& \phantom{...}+  \partial_\#(\theta_\#(\theta_\#(T_1 \cdot C_{123}^*) \cdot \bar \phi(\mu_4^*)))\\
&=& \# \Big( |T_1| \cap |F_1| \cap|C_{234}| \Big) \\
&\phantom{........}& \phantom{..}+ \# \Big(|T_1| \cap |C_{12}| \cap |C_{34}| \Big)\\ 
&\phantom{........}& \phantom{.....}+ \# \Big(|T_1| \cap |C_{123}| \cap |F_{4}| \Big),
\end{eqnarray*} 
where $ C_{12}^* = \overline \phi ( \overset  \thicksim  C_{12}), C_{34}^* = \overline \phi ( \overset  \thicksim  C_{34}), C_{123}^* = \overline \phi ( \overset  \thicksim  C_{123}), C_{234}^* = \overline \phi ( \overset  \thicksim  C_{234})$, and $ C_{12}=(\theta \circ r \circ i)_\#(C_{12}^*), C_{34}= (\theta \circ r \circ i)_\#(C_{34}^*), C_{123}=(\theta \circ r \circ i)_\#(C_{123}^*), C_{234}=(\theta \circ r \circ i)_\#(C_{234}^*)$.  
\end{theorem}

Geometric interpretations of $C_{123}$ and $C_{234}$ can be made similar to that of $C_{12}$ and $C_{23}$ in the case of third-order linking.   For $C_{123}$, we have

\begin{eqnarray*}
( r_* \circ i_*) ([ \overline \phi (\mu_1^* \cup  \overset \thicksim  C_{23} +  \overset \thicksim  C_{12} \cup \mu_3^*  )]  &=& \theta_*(i_*([ \bar \phi (\mu_1^*)])) \cdot [\bar \phi (\overset \thicksim  C_{23})] +\theta_*(i_*([\bar \phi(\overset \thicksim  C_{12})]) \cdot [\bar \phi  (\mu_3^*)])] \\ &=& [\theta _\#(i_\#(F_1)) \cdot C_{23}^* + \theta _\#(i_\#(C_{12}^*)) \cdot F_3],
\end{eqnarray*} 
and the left hand side is equal to

\begin{eqnarray*}
( r_* \circ i_*) ([  \bar \phi ( \delta \overset  \thicksim C_{123})] 
&=& ( r_* \circ i_*) ([ \partial  \bar \phi ( \overset  \thicksim C_{123})] \\ 
&=& ( r_* \circ i_*)([\partial C_{123}^*])\\ 
&=& [\partial  ( r \circ i)_\#(C_{123}^*)].\\
\end{eqnarray*}  
Thus we have $\partial  ( r \circ i)_\#(C_{123}^*)   = \theta _\#(i_\#(F_1)) \cdot C_{23}^* + \theta _\#(i_\#(C_{12}^*)) \cdot F_3\phantom{.} + \phantom{.}b_{123}'$, where $b_{123}' \phantom{.} \epsilon \phantom{.} B_1({S^3}', L')$.  Applying $\theta_\#$ on both sides we obtain $\partial ( ( \theta \circ r \circ i)_\#(C_{123}^*) )= \theta _\#((\theta_\# ( r_\#(i_\#(F_1))) \cdot C_{23}^* \phantom{.}+  \theta _\#(\theta _\#(i_\#(C_{12}^*)) \cdot F_3)\phantom{.}+\phantom{.} b_{123} $, where $ b_{123} \phantom{.} \epsilon \phantom{.} B_1(S^3, L)$.  Since
$b_{123} = \partial \beta_{123} + \overset  \thicksim b_{123}$,  with $\beta_{123} \phantom{.} \epsilon \phantom{.} C_2((S^3)')$ and $\overset  \thicksim b_{123} \phantom{.} \epsilon \phantom{.} C_2(L')$, and  since ($ r \circ i)_\#$ is the "retraction" from $((S^3)^*, L^*)$ to $(S^3,L)$,
 $C_{123}=( \theta \circ r \circ i)_\#(C_{123}^*) $ can be viewed as a 2-complex in $(S^3,L)$ whose boundary, after attaching to the boundary of some 2-complex in $(S^3,L)$, is the sum of the intersection of $F_1$(which is identified with $( \theta \circ r \circ i)_\#(F_1)$) and $C_{23}$ and the intersection of $C_{12}$ and $F_3$, rel $L$.  One such candidate for $C_{123}$ can be taken to be  any spanning surface for the new link components formed by the arcs and circles of $F_1 \cdot C_{23}$ and  $C_{12} \cdot F_3$ and portions of components of $L$.  Similar choice can be made for $C_{234}$, which is any spanning surface for the new link components formed by arcs and circles of $F_2 \cdot C_{34}$ and  $C_{23} \cdot F_4$ and portions of components of $L$.  Thus $C_{123}$ and $C_{234}$ are reductively constructed from $C_{12}$, $C_{23}$, and $C_{34}$, which are constructed from the previous case of third-order linking.
 
\phantom{.} 

\emph{Remark.} A formula for Massey's linking number of order $\ge 5$ can be obtained inductively when its linking numbers of order $\le n-1$ all vanish, following the pattern in the third- and fourth-order linkings.  For example, the fifth order linking number is given by 

\phantom{.} 

$[\mu_1^* \cup \overset  \thicksim  C_{2345} + \overset  \thicksim  C_{12} \cup \overset  \thicksim  C_{345} + \overset  \thicksim  C_{123} \cup \overset  \thicksim  C_{45} + \overset  \thicksim  C_{1234} \cup \mu_5^*] ([T_1])$

\begin{eqnarray*}
&=& \partial_\#(\theta_\#(\theta_\#(T_1 \cdot \bar \phi(\mu_1^*)) \cdot C_{2345}^*) \\
&\phantom{....}& \phantom{.}+\partial_\#(\theta_\#(\theta_\#(T_1 \cdot C_{12}^*) \cdot C_{345}^*)) \\
&\phantom{....}& \phantom{...}+\partial_\#(\theta_\#(\theta_\#(T_1 \cdot C_{123}^*) \cdot C_{45}^*)) \\
&\phantom{....}& \phantom{.....}+  \partial_\#(\theta_\#(\theta_\#(T_1 \cdot C_{1234}^*) \cdot \bar \phi(\mu_5^*)))\\
&=& \# \Big( |T_1| \cap |F_1| \cap|C_{2345}| \Big) \\
&\phantom{........}& \phantom{..}+ \# \Big(|T_1| \cap |C_{12}| \cap |C_{345}| \Big)\\ 
&\phantom{........}& \phantom{.....}+ \# \Big(|T_1| \cap |C_{123}| \cap |C_{45}| \Big)\\
&\phantom{........}& \phantom{.......}+ \# \Big(|T_1| \cap |C_{234}| \cap |F_{5}| \Big),
\end{eqnarray*} 
where $ \overset  \thicksim  C_{12}, C_{12} , \overset  \thicksim  C_{45}, C_{45}, \overset  \thicksim  C_{123}, C_{123}, \overset  \thicksim  C_{345}$ and $C_{345}$ are similarly obtained as before, and $\overset  \thicksim  C_{1234}$ and $ \overset  \thicksim  C_{2345}$ satisfy $\mu_1^* \cup \overset  \thicksim  C_{234} + \overset  \thicksim  C_{12} \cup \overset  \thicksim  C_{34} + \overset  \thicksim  C_{123} \cup \mu_4^* = \delta \overset  \thicksim  C_{1234}$ and $\mu_2^* \cup \overset  \thicksim  C_{345} + \overset  \thicksim  C_{23} \cup \overset  \thicksim  C_{45} + \overset  \thicksim  C_{234} \cup \mu_5^* = \delta \overset  \thicksim  C_{2345}$. 

\section{A combinatorial algorithm for computing Massey numbers} 
  
  In this section we give an algorithm for computing the third-order linking number by first constructing $\partial C_{12}$ and $\partial C_{23}$, which will be oriented links in $R^3$ and from which $C_{12}$ and $C_{23}$ can be constructed as the Seifert surfaces spanning $\partial C_{12}$ and $\partial C_{23}$, respectively.  We will see that this procedure can be generalized to the 4th and higher order linkings.  The construction will be facilitated by the use of Seifert disks for $F_1$ and $F_2$, which we discuss in the following.
  
  \phantom{.}
  
  Let $S^3$ be given the usual right-hand orientation, and let $K_1$, $K_2$ be oriented knots with oriented Seifert surfaces $F_1$ and $F_2$ respectively.  Then $F_1$ and $F_2$ are each a union of Seifert disks and half-twist bands.  We may assume that they are in general position.

\phantom{.}

Consider $F_1 \cdot F_2$.  By an isotopy we may assume that the Seifert disks of $F_1$ intersect only with Seifert disks of $F_2$ and vice versa.  This can be done by pushing Seifert disks of $F_1$ away from the twist bands of $F_2$, and pushing Seifert disks of $F_2$ away from the twist bands of $F_1$.  By using an innermost disk argument, we may also assume that no curve of intersection of any two Seifert disks is a circle.  In this case, we say that $F_1$ and $F_2$ are in {\it normal position}.  When $F_1$ and $F_2$ are in normal position, a component of  $F_1 \cdot F_2$ is either a curve joining $K_1$ and $K_2$ or a curve joining two points in $K_1$ or two points in $K_2$.  See the figure given below. 

\begin{figure}[h]
	\centering
		\scalebox{0.5}{\includegraphics{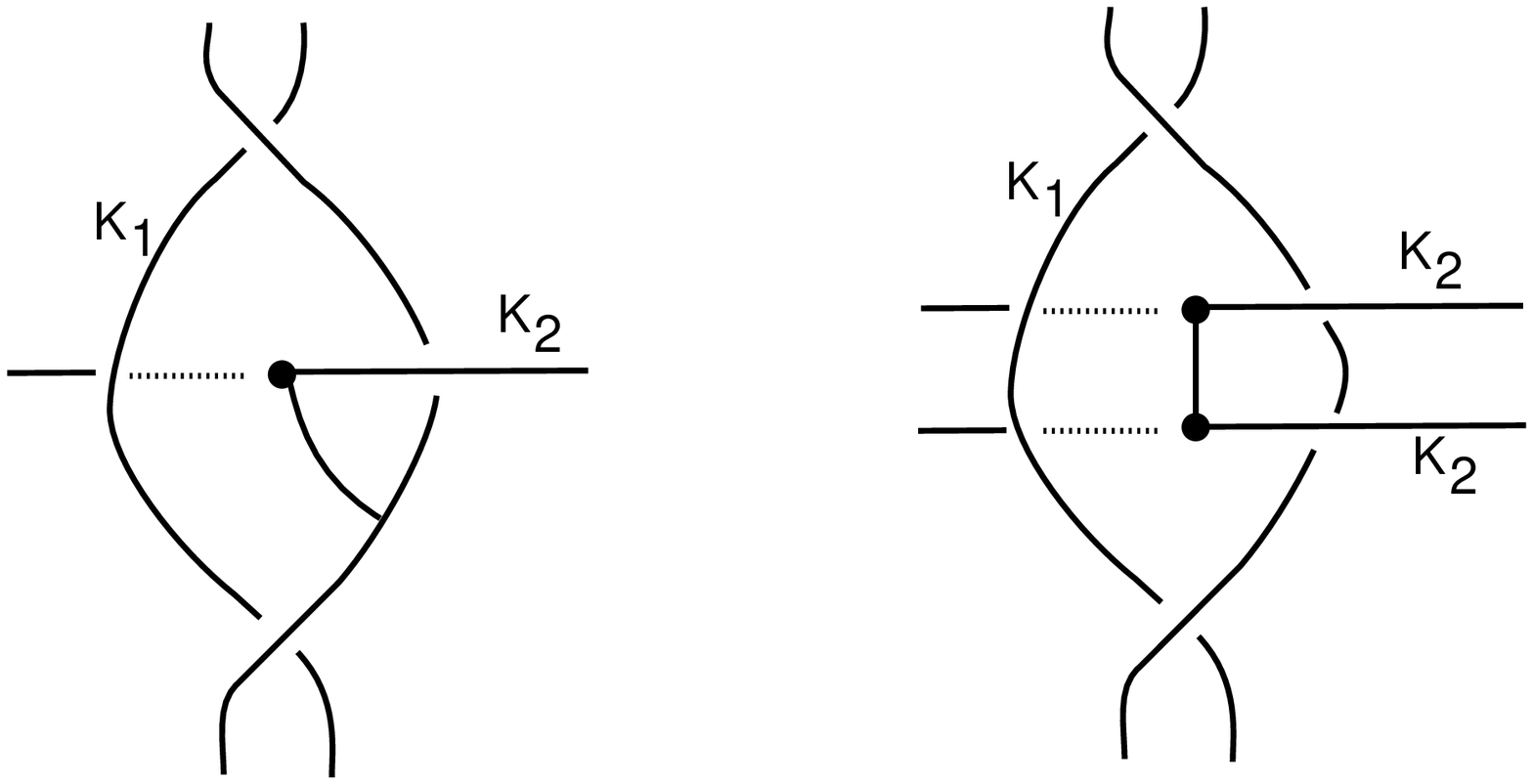}}
	\label{fig:Figure 1}

\end{figure}

The intersection theory at chain-cochain level developed earlier can be applied to show that, a curve of intersection of a Seifert disk of $F_1$ and a Seifert disk of $F_2$ has the following orientation determined by that of $F_1$ and $F_2$:  For $i=1,2$, let $d_i$ be a Seifert disk of $F_i$ and let $\alpha$ be a curve of $d_1 \cdot d_2$ (in this order).  The orientation on $\alpha$ is the one satisfying the diagramatic convention depicted in the following figures.

\phantom{.}

\newpage

(1)

\begin{figure}[h]
	\centering
		\scalebox{0.5}{\includegraphics{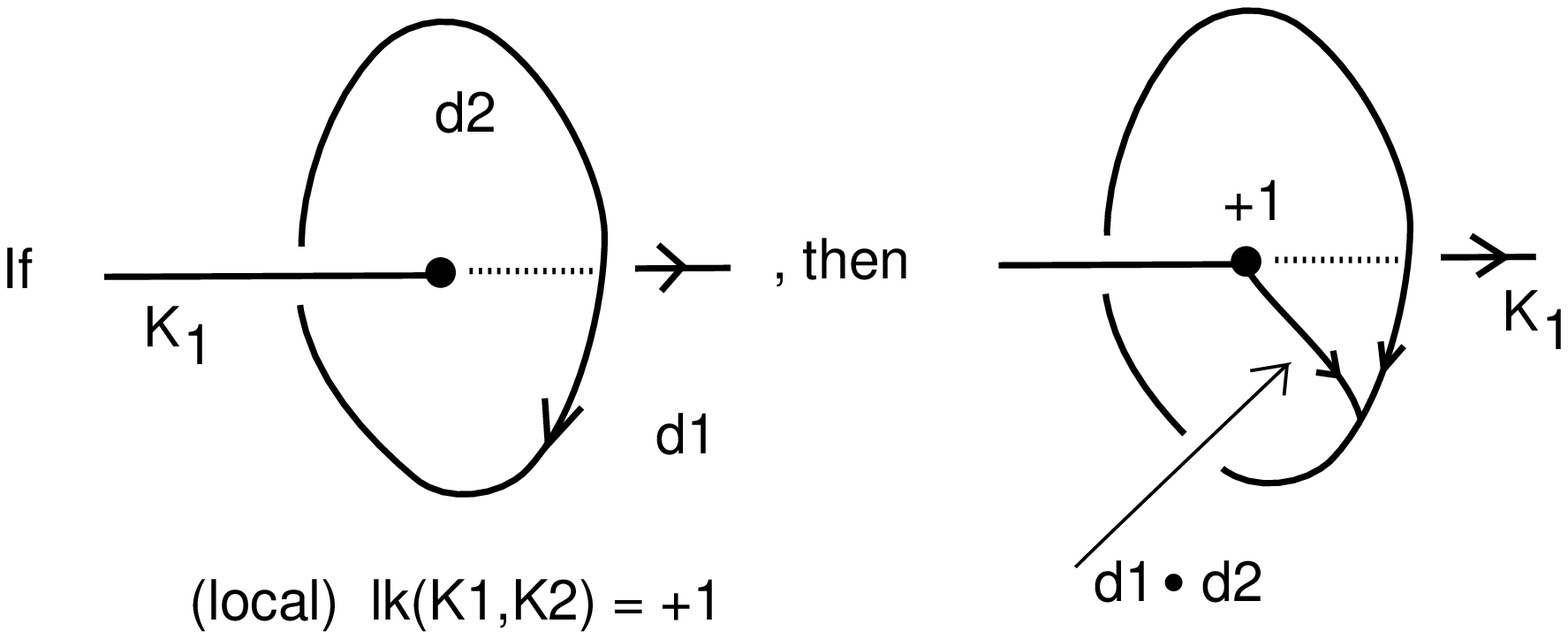}}
	\label{fig:Figure 2}
	
\end{figure}

That is, if the local linking of $K_1$ and $K_2$ is $+1$, then the curve $d_1 \cdot d_2$ (in this order) goes "out" from the point of intersection of $K_1$ and $d_2$ labeled with $+1$ to the boundary of $d_2$.  Note that the disk $d_1$ in the diagram can be on either side of $K_1$.

\phantom{.}

(2)

\begin{figure}[h]
	\centering
		\scalebox{0.5}{\includegraphics{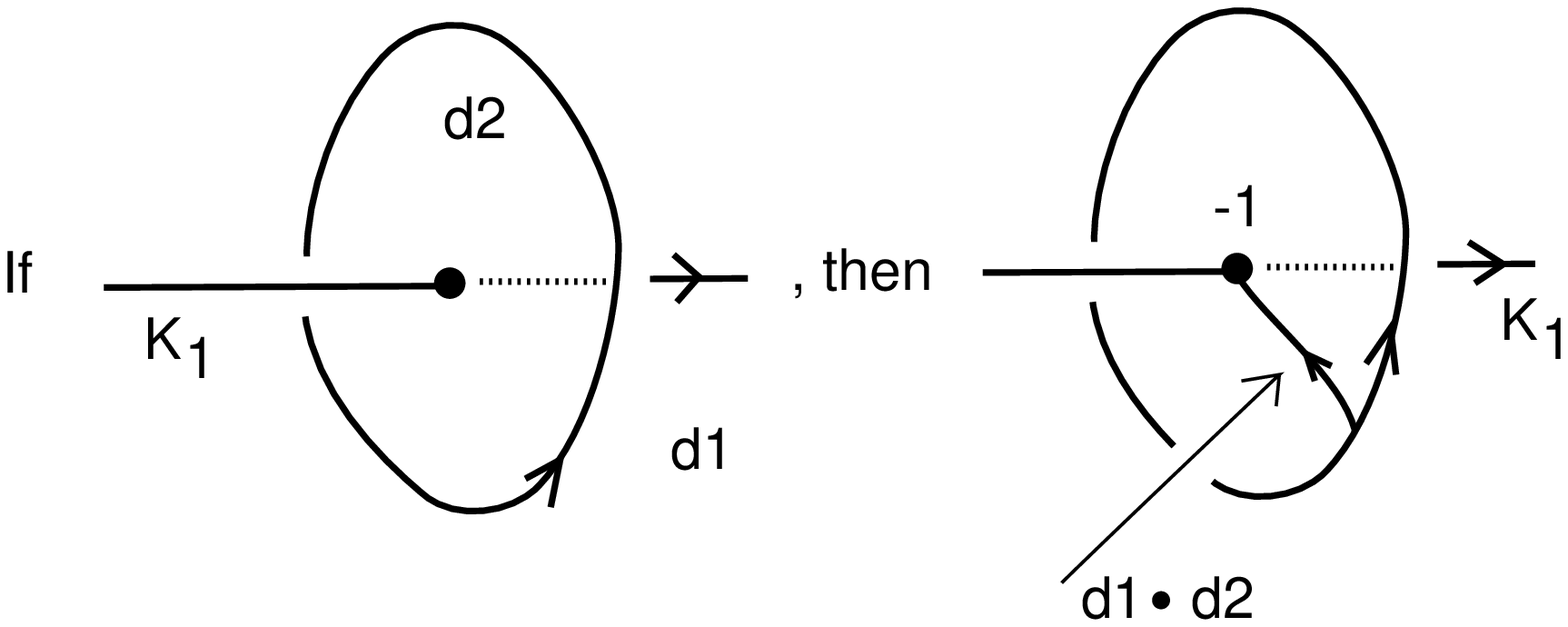}}
	\label{fig:Figure 3}
	
\end{figure}

Thus if the local linking of $K_1$ and $K_2$ is $-1$, then the curve $d_1 \cdot d_2$ (in this order) goes 
"into"  the point of intersection of $K_1$ and $d_2$ labeled with $-1$ from a boundary point of $d_2$.

\phantom{.}  

(3)

\begin{figure}[h]
	\centering
		\scalebox{0.5}{\includegraphics{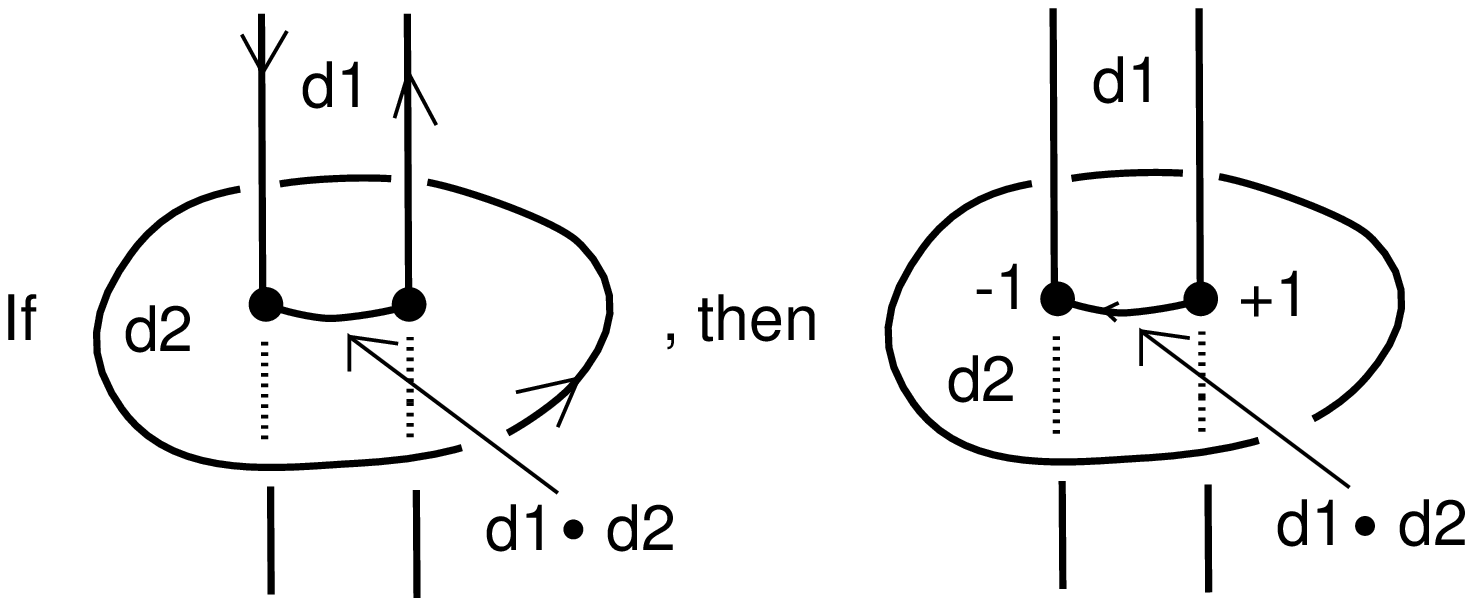}}
	\label{fig:Figure 4}
	\end{figure}

\phantom{.}

  In this case $d_1 \cdot d_2$ goes from the point of intersection labeled with $+1$ to the point of intersection labeled with $-1$.  Note that when considering $F_1 \cdot F_2$ we need only consider Seifert disks of $F_2$, and see how $K_1$ intersects these disks.
  
\phantom{.}
  
  Now Massey's Third-order linking number is $<[\mu_1^*],[\mu_2^*],[\mu_3^*]> ([T_1]) = [\mu_1^* \cup \overset \sim C_{23} + \overset \sim C_{12} \cup \mu_3^*]([T_1]) = \partial_\#(\theta_\#(T_1 \cdot F_1) \cdot C_{23}) + \partial_\#(\theta_\#(T_1 \cdot C_{12}) \cdot F_3)$, where $C_{12} \phantom{.} \epsilon \phantom{.} C_2(S^3, K_1 \cup  K_2)$ with $|\partial C_{12}| \subseteq |(F_1 \cdot F_2 )| \cup |K_1 \cup K_2|$, and $C_{23} \phantom{.} \epsilon \phantom{.}C_2(S^3, K_2 \cup K_3)$ with $\partial |C_{23}| \subseteq |(F_2 \cdot F_3)| \cup |K_2 \cup K_3|$. 
  
\phantom{.}    
   
 The following procedure gives the construction of $\partial C_{12}$, each component of which is a simple closed curve in $R^3$, from which $C_{12}$ can be constructed as the Seifert surface spanning $\partial C_{12}$.  The same procedure can be used to construct $\partial C_{23}$ and therefore $C_{23}$.  Assume $F_1$ and $F_2$ are in normal position.  Practically we will only need the points of intersection of $K_1$ with Seifert disks of $F_2$ with their $\pm1$ labeling.
  
\phantom{.}
  
  Note that $lk(K_1,K_2)=0$ implies that the number of geometric intersections of $K_1$ with Seifert disks of $F_2$ is even, and the algebraic intersection number is 0.
  
\phantom{.}
  
  (1)  As a starting point, choose $p$ to be any point of intersection of $K_1$ with a Seifert disk of $F_2$ that is labeled $-1$, there are two possibilities, see Figure 1(a) and 1(b).
    
\begin{figure}[h]
	\centering
		\scalebox{0.5}{\includegraphics{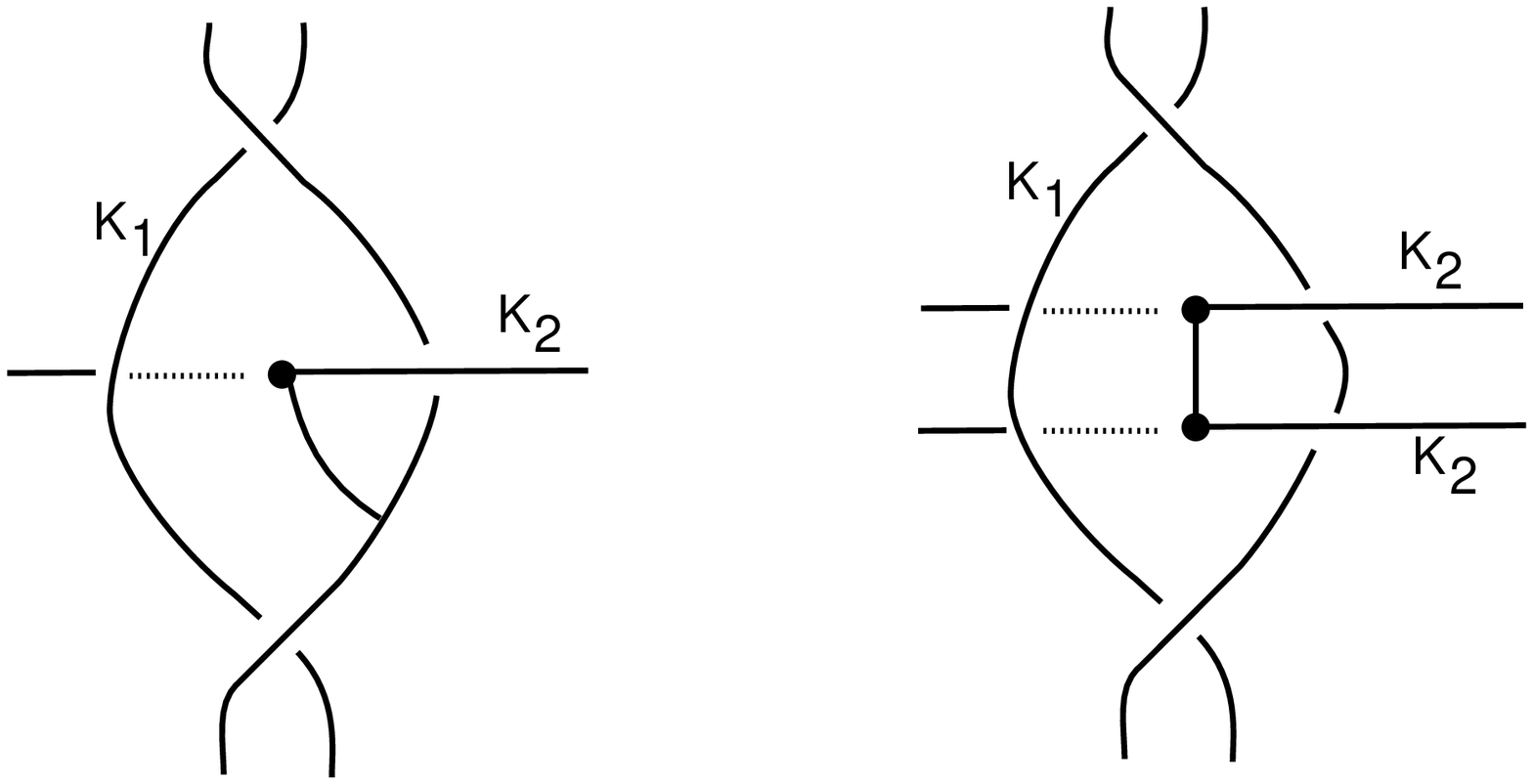}}
	\label{fig:Figure 1}
	\caption*{Figure 1}
\end{figure}

  (2)  Follow the orientation of $K_1$ until meeting the next Seifert disk of $F_2$.
If the intersection point is labeled with $-1$, skip and continue traveling along $K_1$ until meeting the next Seifert disk of $F_2$. 
  
Continue this process (repeatedly going back to (2) if the intersecting point encountered is labeled with $-1$) until an interesection point labeled with $+1$ is encountered.  Such an intersection point exists since the algebraic intersection number of $K_1$ and Seifert disks of $F_2$ is 0.  Then there is an oriented arc of $F_1 \cdot F_2$ contained in this Seifert disk of $K_2$.
  
\phantom{.}  
  
 (3)  Follow the orientation of the oriented arc to reach either
 \phantom{.} 
  
 (3.1) another point of intersection of $K_1$ with the same Seifert disk of $F_2$ labeled with $-1$, see Figure 2;  or
 
\begin{figure}[h]
	\centering
		\scalebox{0.5}{\includegraphics{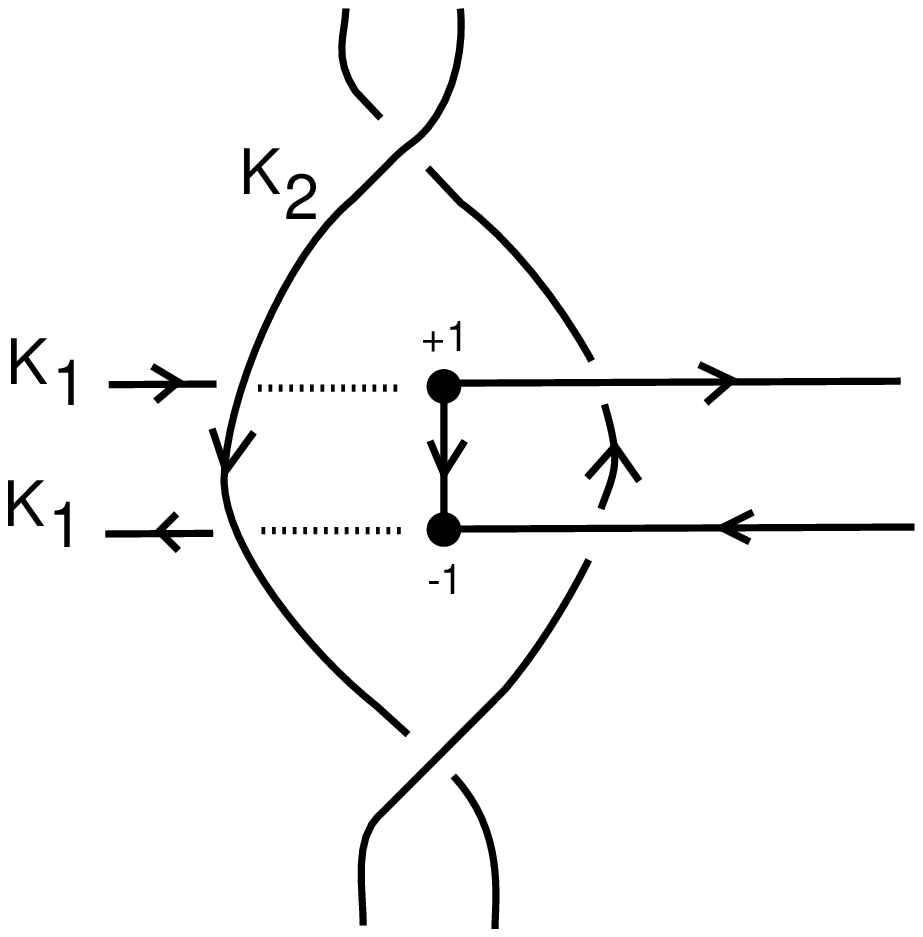}}
	\label{fig:Figure 5}
	\caption*{Figure 2}
\end{figure}

 (3.2) a point in $K_2$, see Figure 3.

\begin{figure}[h]
	\centering
		\scalebox{0.5}{\includegraphics{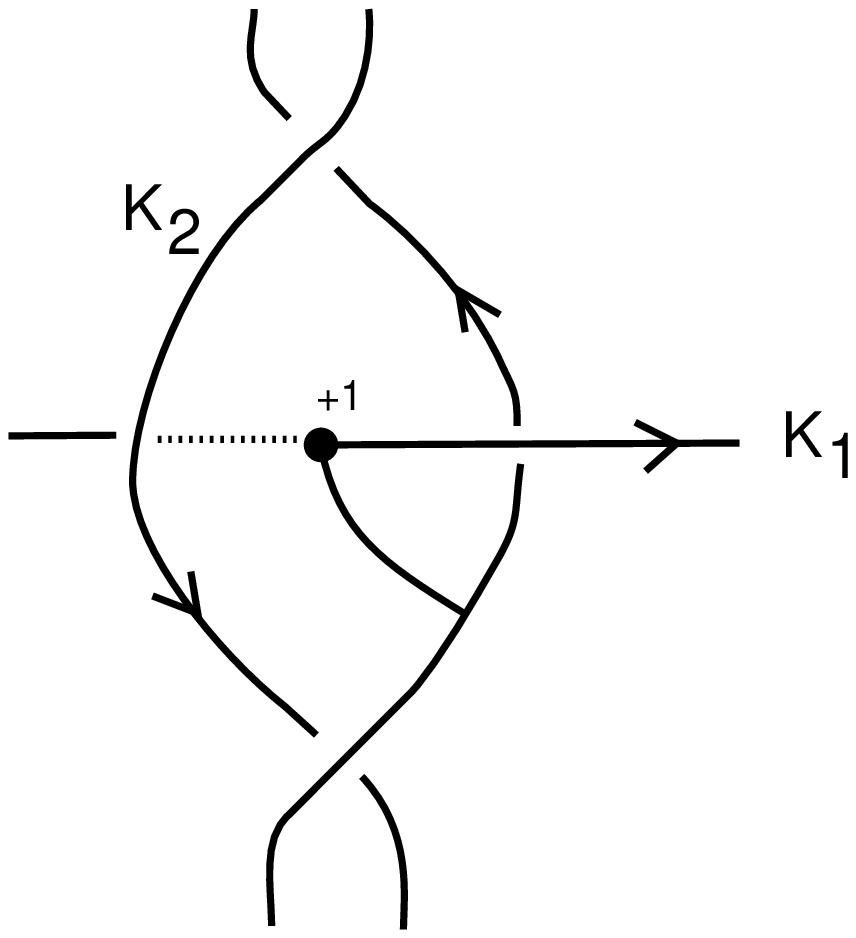}}
	\label{fig:Figure 6}
	\caption*{Figure 3}
\end{figure}

In case of (3.1), go back to (1) and continue the procedure, and proceeds with the condition that when a $+1$ intersection point is encountered, it needs be checked whether the point has been encountered previously,  skip and continue if encountered previously.
 
 \phantom{.} 
 
 In case of (3.2),  follow the orientation of $K_2$, until it either
 
 \phantom{.}  
 
 (3.2.1)  reaches another intersection arc of $F_1 \cdot F_2$ lying in this Seifert disk of $F_2$, see Figure 4; or
 
 \newpage
 
\begin{figure}[h]
	\centering
		\scalebox{0.5}{\includegraphics{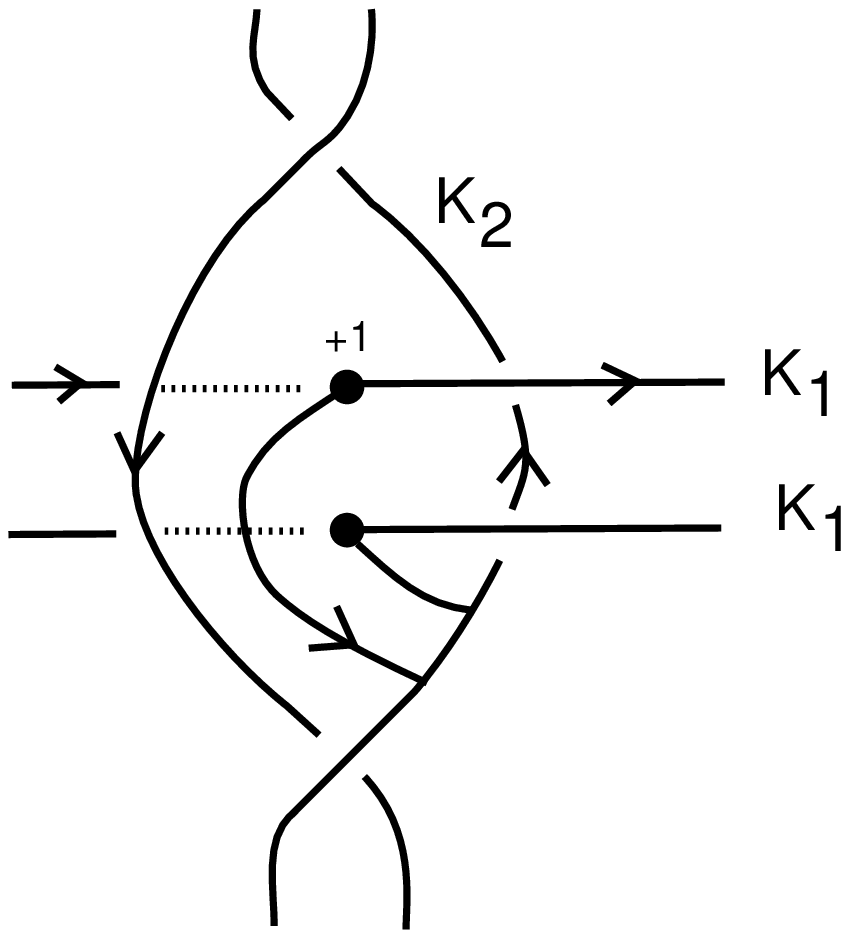}}
	\label{fig:Figure 7}
	\caption*{Figure 4}
\end{figure}

 (3.2.2)  reaches a boundary component of a twist band attached to this Seifert disk of $F_2$, see Figure 5.
 
\begin{figure}[h]
	\centering
		\scalebox{0.5}{\includegraphics{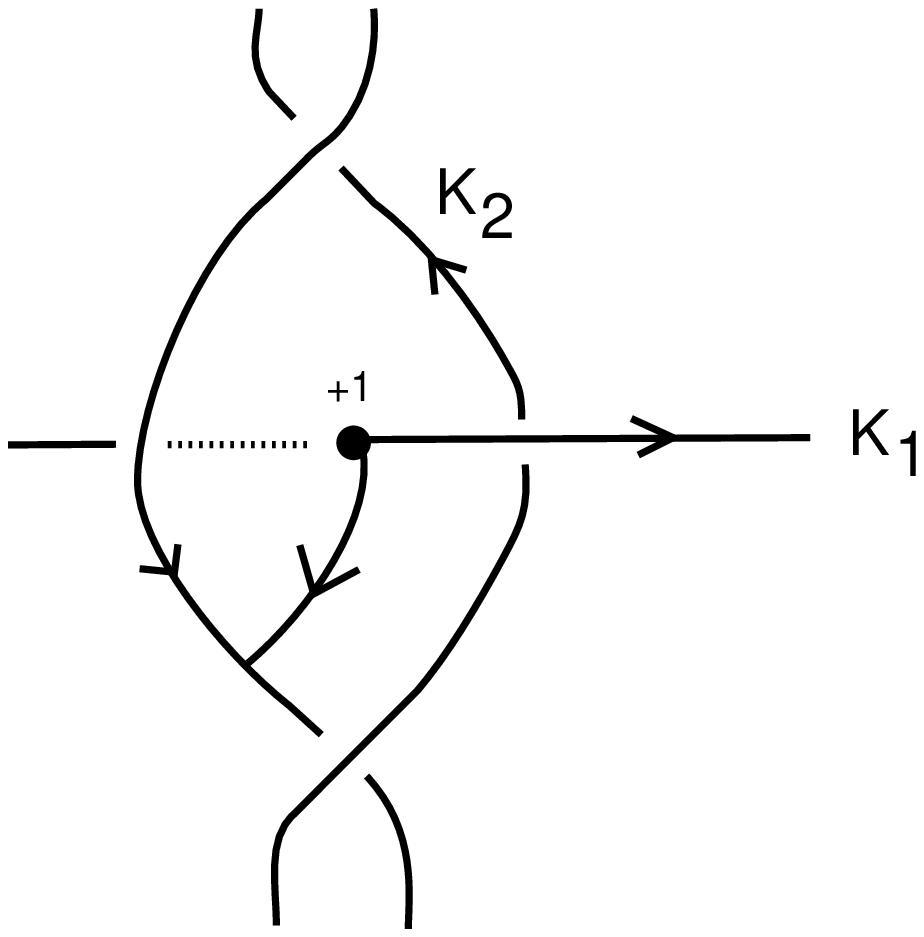}}
	\label{fig:Figure 8}
	\caption*{Figure 5}
\end{figure}

In the case of (3.2.1), check the orientation of the arc of intersection of $F_1 \cdot F_2$ encountered to see if it is possible to follow the orientation.  If not ( this means that the point of intersection for this arc is also labeled with $+1$), then skip and continue on $K_2$ and go back to (3.2), see Figure 6. 

\begin{figure}[h]
	\centering
		\scalebox{0.5}{\includegraphics{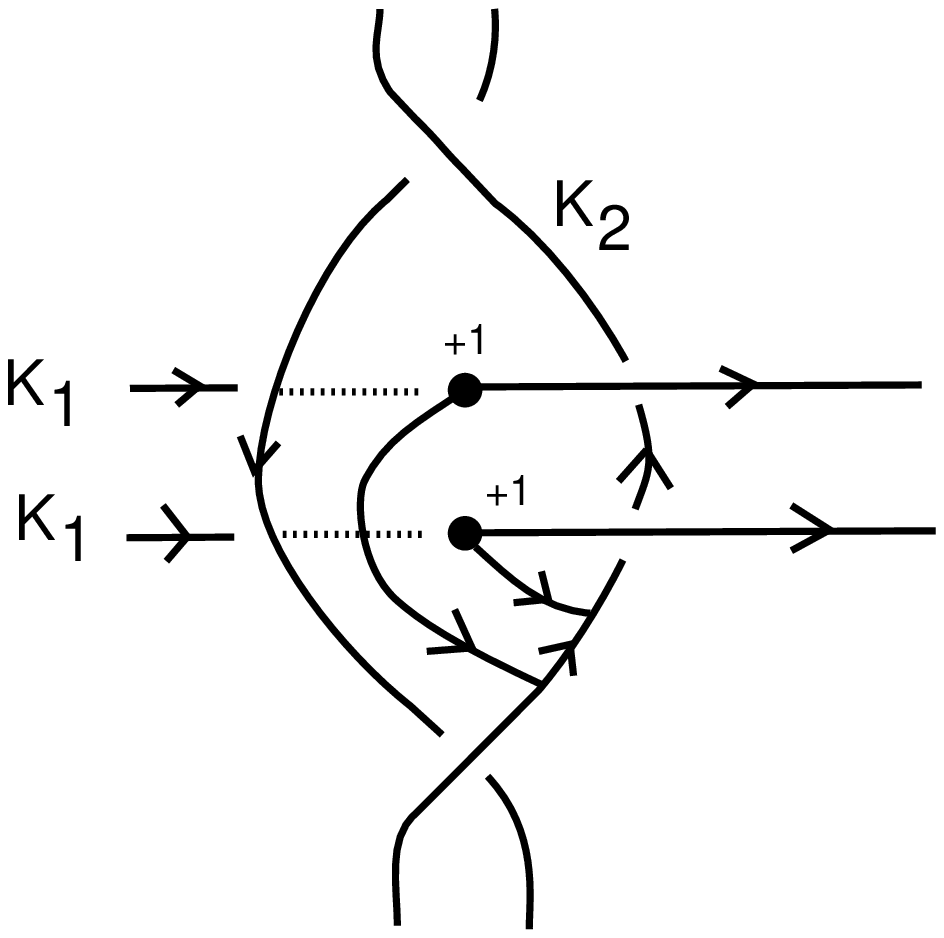}}
	\label{fig:Figure 9}
	\caption*{Figure 6}
\end{figure}

\newpage 

Otherwise follow the orientation of the intersection arc to reach to the point of intersection of $K_1$ with this Seifert disk of $F_2$, see Figure 7.

\begin{figure}[h]
	\centering
		\scalebox{0.5}{\includegraphics{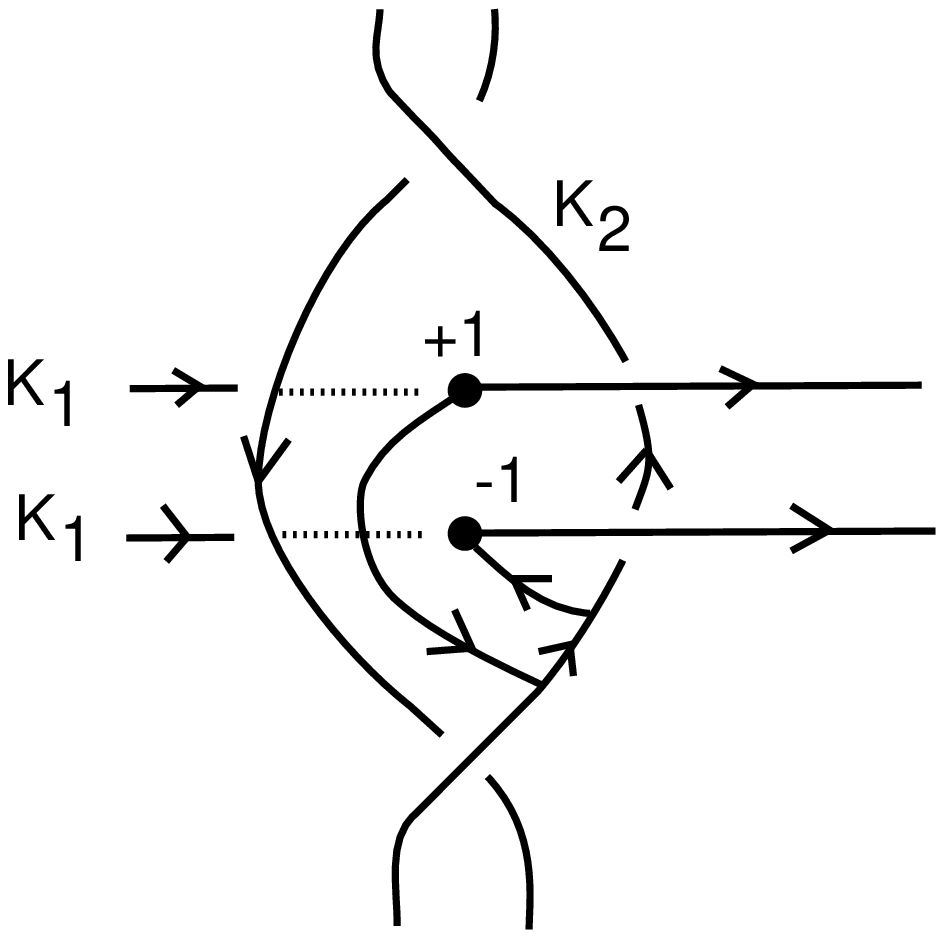}}
	\label{fig:Figure 10}
	\caption*{Figure 7}
\end{figure}

The point of intersection  must be labeled $-1$ , so  

\phantom{.}  

(3.2.1.1) check if it is the point $p$ started initially.

\phantom{.}

If no, go back to (1), and if yes, a simple loop is obtained; then check to see if there are other points of intersection of $K_1$ and Seifert disks of $F_2$ left undone, and if yes, choose any one of such points  and go back to (1) or (3) depending on the label of the point is $-1$ or $+1$ respectively; otherwise all the points of intersection of $K_1$ and Seifert disks of $F_2$ have been accounted for, so $\partial C_{12}$ is obtained and algorithm stops.
  
 \phantom{.} 
 
In the case of (3.2.2), follow the boundary component encountered to reach to a Seifert disk of $F_2$, which may be the one started with initially.  Traveling along $K_2$ in the boundary of this Seifert disk of $F_2$.  Then it either 

\phantom{.} 

(3.2.2.1) meets the boundary component of another twist band attached to this Seifert disk of $F_2$, and if so go back to (3.2.2), see Figure 8; or 

\begin{figure}[h]
	\centering
		\scalebox{0.4}{\includegraphics{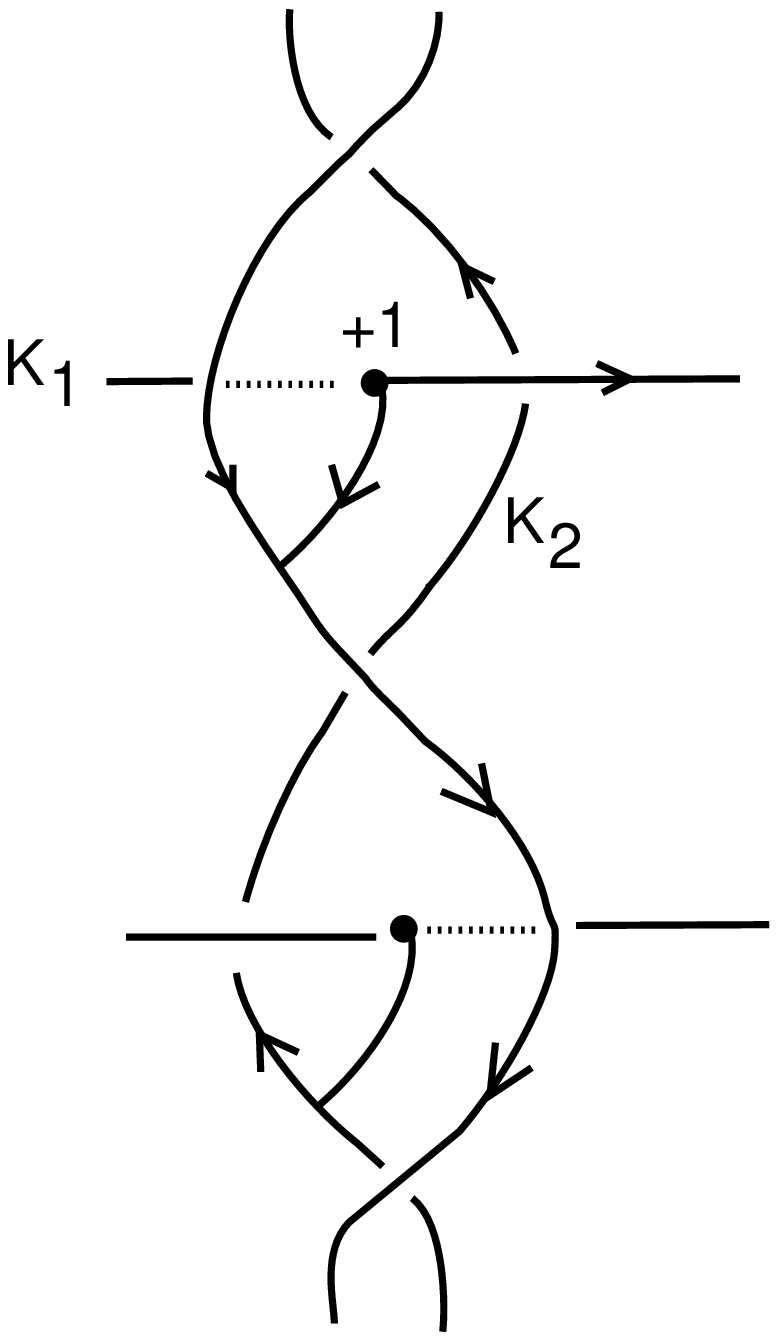}}
	\label{fig:Figure 11}
	\caption*{Figure 8}
\end{figure}

(3.2.2.2) meets an arc of intersection whose orientation cannot be followed, this is equivalent to the case that the point of intersection for this arc is labeled with +1 also, and if so then skip and continue to follow the orientation of $K_2$, i.e. go back to (3.2), see Figure 9; or 

\begin{figure}[h]
	\centering
		\scalebox{0.4}{\includegraphics{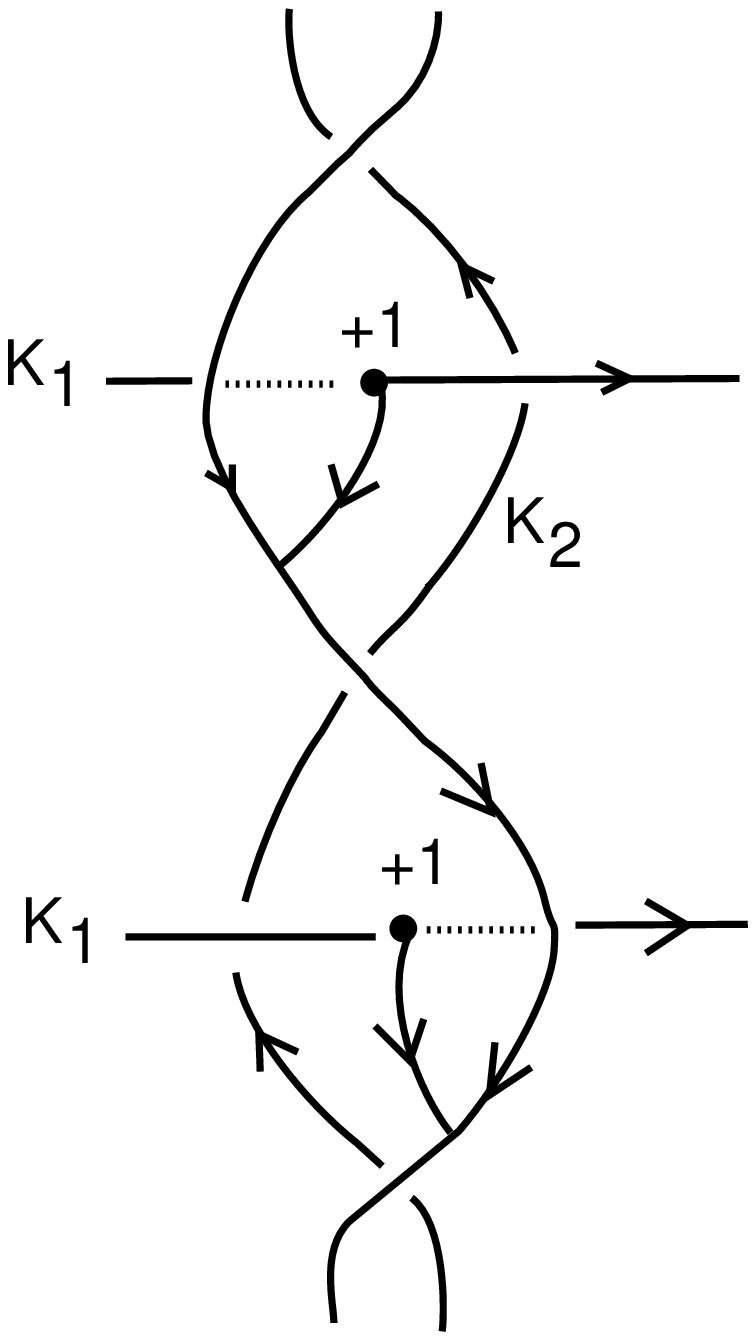}}
	\label{fig:Figure 12}
	\caption*{Figure 9}
\end{figure}

 \newpage

(3.2.2.3) meets an arc of intersection whose orientation can be followed, and this is equivalent to the case that the point of intersection for this arc is labeled with -1.  Then go to (3.2.2.1) to check if the point is the point $p$ started initially, and proceeds accordingly.  See Figure 10

\begin{figure}[h]
	\centering
		\scalebox{0.4}{\includegraphics{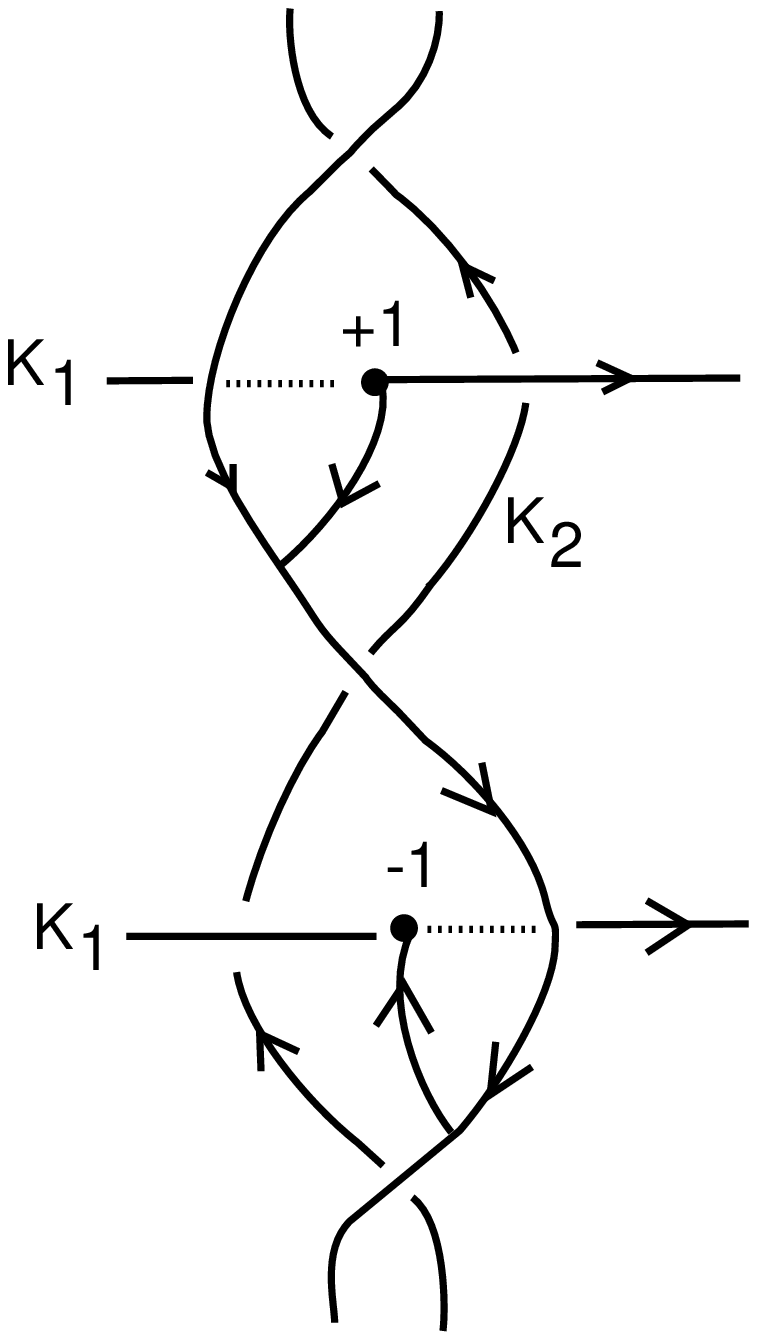}}
	\label{fig:Figure 13}
	\caption*{Figure 10}
\end{figure}

\emph{Remark.}  One can choose the starting point $p$ to be a point of intersection of $K_1$ with a  Seifert disk of $F_2$ labeled with $+1$ instead of $-1$, but then the procedure changes accordingly.  In general, to obtain $\partial C_{12}$ say, first find all the curves of intersection of $F_1 \cdot F_2$ (in this order) and start with any one such curve and following its orientation and the orientations of $L$ and other curves of intersection, until all the curves of intersection have been encountered.

\phantom{.} 

\emph{Example 1.} Let $L = K_1 \cup K_2 \cup K_3$ be the Borromean rings with spanning surfaces $F_1, F_2, F_3$ as depicted in the figure given below.

\begin{figure}[h]
	\centering
		\scalebox{0.4}{\includegraphics{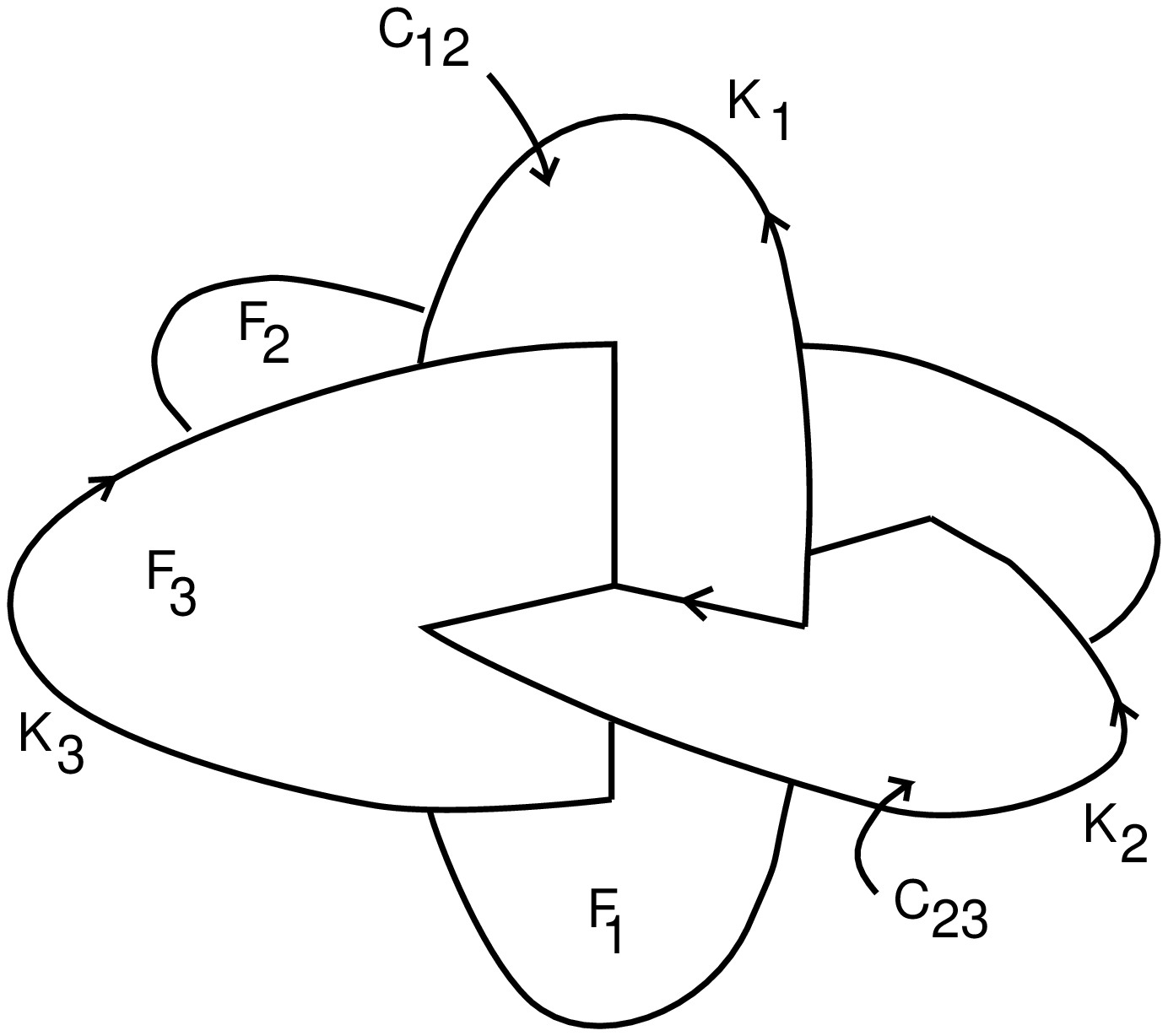}}
	\label{fig:Figure 14}
\end{figure}
Clearly the third-order linking number is 

\begin{eqnarray*}
\# \Big( |T_1| \cap |F_1| \cap|C_{23}| \Big) + \# \Big(|T_1| \cap |C_{12}| \cap |F_3| \Big)
&=& lk(K_1,\partial C_{23}) + \# \Big(|T_1| \cap |C_{12}| \cap |F_3| \Big)\\
&=& 1 + 0\\
&=& 1.  
\end{eqnarray*}

\phantom{.}  

\emph{Example 2.} Consider the link $L = K_1 \cup K_2 \cup K_3$  as depicted in the figure given below

\begin{figure}[h]
	\centering
		\scalebox{0.4}{\includegraphics{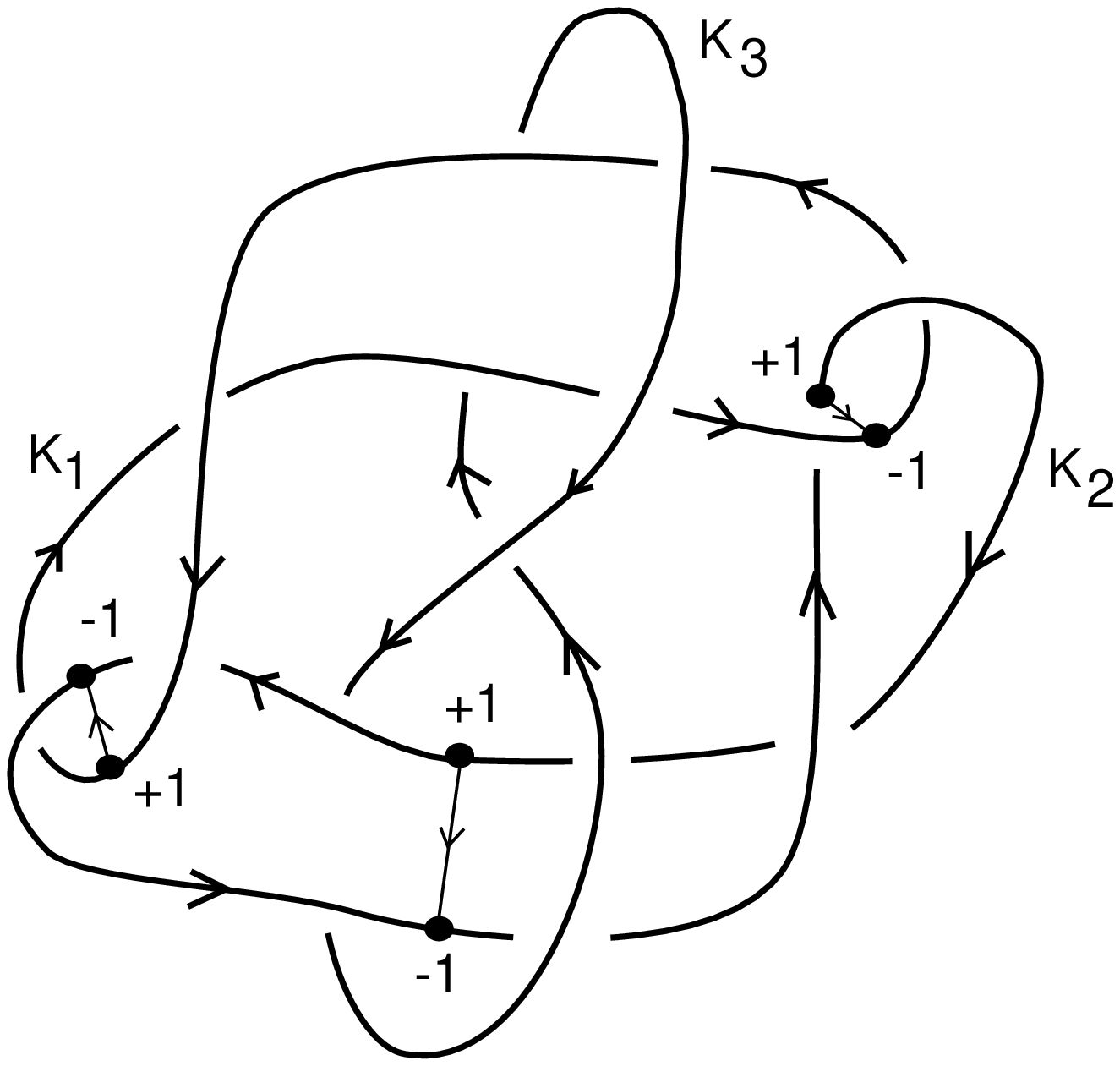}}
	\label{fig:Figure 16}
\end{figure}
Here again the Seiert surface $F_i$ for $K_i$ is a disk, and one computes

$\# \Big( |T_1| \cap |F_1| \cap|C_{23}| \Big) + \# \Big(|T_1| \cap |C_{12}| \cap |F_3| \Big)
= lk(K_1,\partial C_{23}) = -1$, see Figure 11;

\begin{figure}[h]
	\centering
		\scalebox{0.4}{\includegraphics{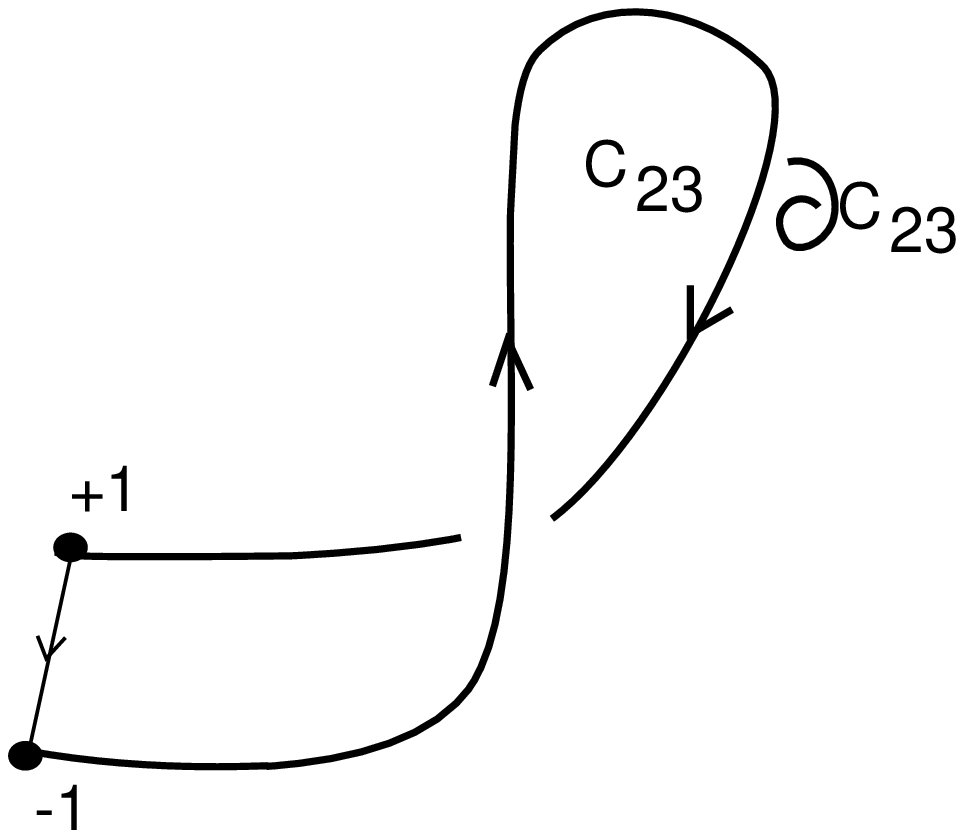}}
	\label{fig:Figure 17}
	\caption*{Figure 11}
\end{figure}

\newpage
and $lk(K_1,\partial C_{23}) + \# \Big(|T_1| \cap |C_{12}| \cap |F_3| \Big) = -2$, see Figure 12. 

\begin{figure}[h]
	\centering
		\scalebox{0.4}{\includegraphics{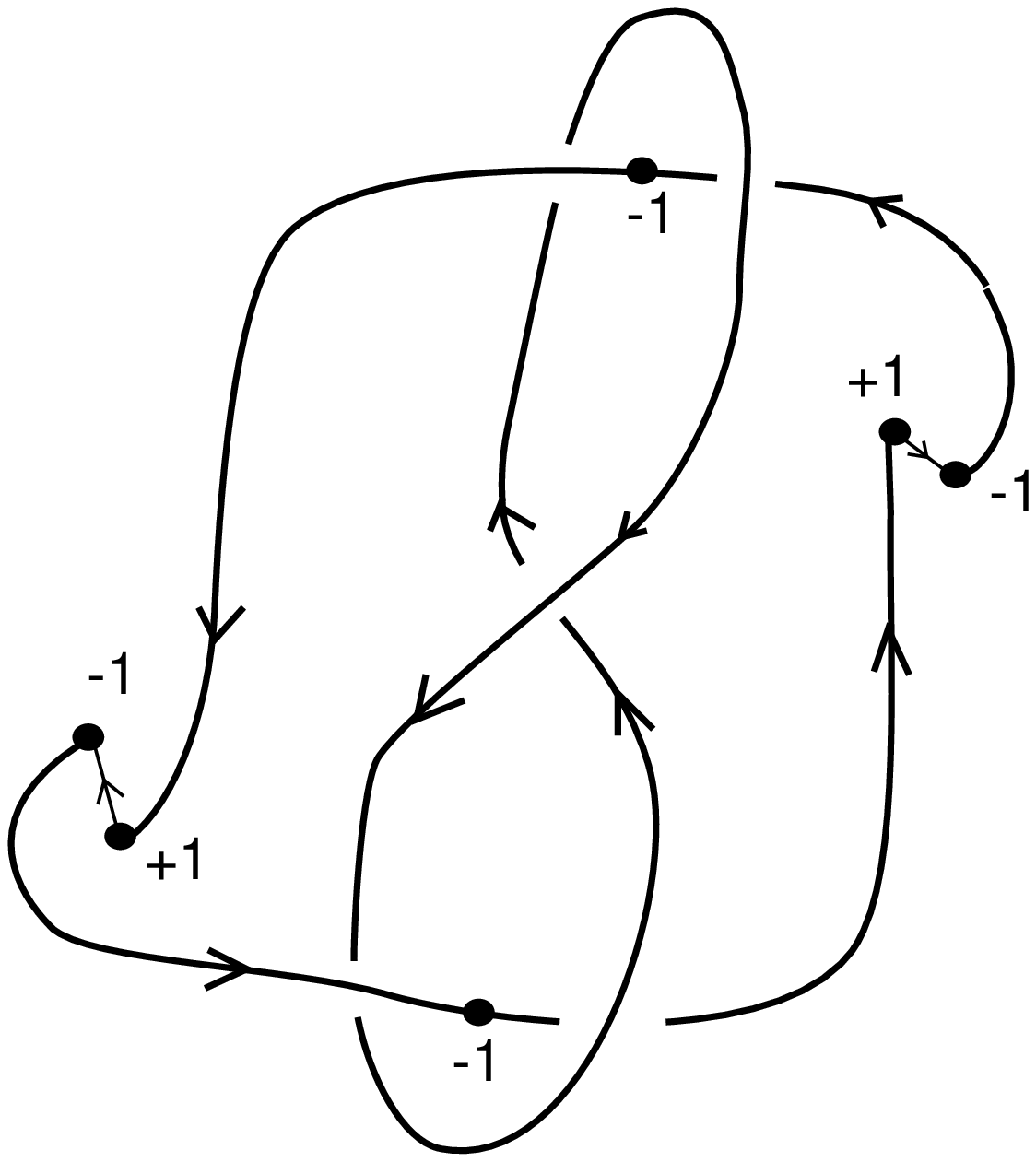}}
	\label{fig:Figure 18}
	\caption*{Figure 12}
\end{figure}

So the third-order linking number is $-3$, which indicates that the components of $L$ are "more linked" than the 
Borromean rings.

\section{Appendix}

The intersection product has the following combinatorial and geometric topology interpretation, which is illustrated for the case $n=3$, $\sigma$ a 2-simplex, and $\tau$ a 1-simplex, see the figure given below. 

\begin{figure}[h]
	\centering
		\scalebox{0.4}{\includegraphics{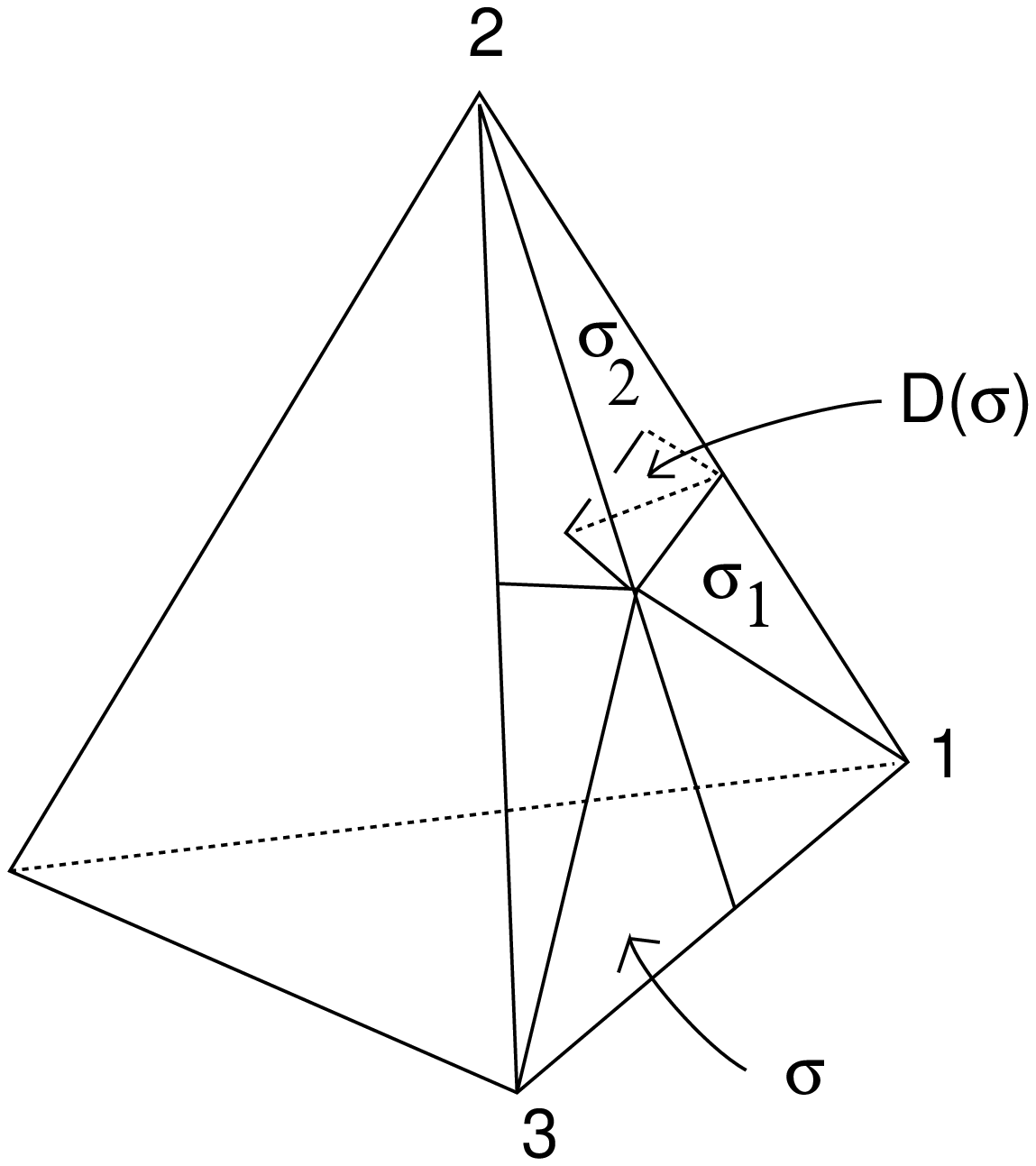}}
	\label{fig:Figure 15}
\end{figure}

Then $\mu_\tau \phantom{.} \epsilon \phantom{.} C^1(K)$, and $D(\tau)\phantom{.} \epsilon \phantom{.} C_2(K^*)$, and by the topological definition of cap product we have 

\begin{eqnarray*}
\sigma \cdot D(\tau) &=& sd_\#(\sigma) \cap \theta^\#(u_\tau) \\
&=& [sd_\#(\sigma) \lambda_1,\theta^\#(u_\tau)] (sd_\#(\sigma) \rho_1),
\end{eqnarray*}  
where $\lambda_1$ is the  front first ace and $\rho_1$ is the back first face.  Let $sd_\#(\sigma)  = \sum_{i=1}^6 \sigma_i$, so 

\begin{eqnarray*}
\sigma \cdot D(\tau) &=& \sum_{i=1}^6 (\theta^\#(\mu_\tau) ( \sigma_i \lambda_1))(\sigma_i \rho_1)\\
&=& \sum_{i=1}^6 \mu_\tau(\theta_\# ( \sigma_i \lambda_1))\\
&=& \sigma_1 \rho_1,
\end{eqnarray*}
since   
\begin{equation} 
 \mu_\tau(\theta_\# ( \sigma_i \lambda_1))= \left\{
  \begin{array}{cl}
  \mu_\tau(\tau) = 1 &\mathrm{for}\quad i = 1\\
 0 &\mathrm{otherwise,}  
  \end{array} \right. 
  \end{equation}
Notice the sense of direction of the 1-simplex $\sigma_i \rho_1$  of $\sigma \cdot D(\tau)$.

\end{document}